\documentclass[10pt,a4paper]{article}
\usepackage{amsmath,amssymb,amsthm,multirow,graphicx,ifpdf,mathtools}
\usepackage{fixltx2e}
\newtheorem{theorem}{Theorem}[section]
\newtheorem{lemma}[theorem]{Lemma}

\newtheorem{corollary}[theorem]{Corollary}

\theoremstyle{definition}
\newtheorem{definition}[theorem]{Definition}
\newtheorem{remark}[theorem]{Remark}
\newtheorem{example}[theorem]{Example}

\numberwithin{equation}{section}

\usepackage{enumitem}

\usepackage{xcolor}
\allowdisplaybreaks

\newcommand{\R}{\mathbb{R}}
\newcommand{\C}{\mathbb{C}}
\newcommand{\N}{\mathbb{N}}

\newcommand{\Q}{\mathbb{Q}}

\begin{document}
\title{Asymptotic behavior of solutions of\\ linear multi-order fractional differential equation systems}
\author{Kai Diethelm\footnote{GNS mbH Gesellschaft f\"ur numerische Simulation mbH, Am Gau\ss berg 2, 38114 Braunschweig, Germany, {\tt diethelm@gns-mbh.com}}\,\,\footnote{AG Numerik, Institut Computational Mathematics, Technische Universit\"at Braunschweig, Universit\"atsplatz 2, 38106 Braunschweig, Germany, {\tt k.diethelm@tu-braunschweig.de}}
\and
Stefan Siegmund\footnote{Center for Dynamics \& Institute for Analysis, Department of Mathematics, Technische Universit\"at Dresden, 01062 Dresden, Germany, {\tt stefan.siegmund@tu-dresden.de}}
\and
H.T.~Tuan\footnote{Institute of Mathematics, Vietnam Academy of Science and Technology, 18 Hoang Quoc Viet, 10307 Ha Noi, Viet Nam, {\tt httuan@math.ac.vn}}
}

\maketitle

\begin{abstract}
In this paper, we investigate some aspects of the qualitative theory for
multi-order fractional differential equation systems. First, we obtain a fundamental
result on the existence and uniqueness for multi-order fractional differential
equation systems. Next, a representation of solutions of homogeneous linear
multi-order fractional differential equation systems in series form is
provided. Finally, we give characteristics regarding the asymptotic behavior
of solutions to some classes of linear multi-order fractional differential
equation systems. 
\end{abstract}



\section{Introduction}

In recent years, fractional calculus has received increasing attention due to
its applications in a variety of disciplines such as mechanics, physics,
chemistry, biology, electrical engineering, control theory, material science,
mathematical psychology. For more details, we refer the reader to the
monographs \cite{BaleanuEtal2016,Diethelm2010,KilbasEtal2006,MillerRoss1993,Podlubny1999}.

A particularly interesting aspect in this connection that does not pertain to classical mathematical models using integer-order differential operators
has recently been discussed in the context of a number of applications in the life sciences \cite{BaleanuEtal2016,Diethelm2013,Diethelm2016}:
It appears that certain real world problems can be described by a system of fractional differential equations where each equation may have 
an order that differs from the orders of the other equations of the system.
We shall call such systems \emph{multi-order fractional differential systems}.

Among the published papers, it seems that the authors mainly concentrated on
approximating solutions of multi-order fractional differential
equations, see e.g.\ \cite{AtabakzadehEtal2013,DiethelmFord2004,El-MesiryEtal2005,GejjiJafari2007,HesameddiniAsadollahifard2015,Li2010,Liu2015,SardarEtal2009,SweilamEtal2007,YangEtal2010,Yang2013}. 
The investigation of the analytical properties of such systems is often restricted to the case where the orders of
the differential operators are rational \cite{DengEtal2007,Diethelm2008,Diethelm2010,ZhangEtal2013}.
For the general case, rigorous mathematical studies of even the most fundamental questions in this context do not seem to be readily available.

Therefore, in this paper we consider $d$-dimensional linear multi-order fractional differential equation systems
\begin{equation}\label{eq:system}
   D_{*}^{\alpha_i} x_i(t)
   =
   \sum_{j = 1}^d a_{ij} x_j(t)
   +
   g_i(t), 
   \qquad
   i = 1, 2, \ldots, d,
\end{equation}
with orders $\alpha_i \in (0,1]$, coefficients $a_{ij} \in \C$, $g_i
\colon [0,\infty) \rightarrow \C$ continuous, $i, j = 1, \dots, d$, and the Caputo
differential operator of order $\alpha > 0$ 
\begin{equation*}
   D_{*}^{\alpha} y(t)
   \coloneqq
   J^{\lceil \alpha \rceil - \alpha} D^{\lceil \alpha \rceil} y(t)
\end{equation*}
which is defined for $C^{\lceil \alpha \rceil}$ functions $y \colon [0, T] \rightarrow \C^d$, $T>0$, with the classical derivative $D$ and the Riemann-Liouville operator 
\begin{equation*}
   J^{\beta} y(t) 
   \coloneqq
   \frac{1}{\Gamma(\beta)} \int_{0}^{t} (t-s)^{\beta - 1} y(s) \,ds
\end{equation*}
for $\beta > 0$ and $J^0 y(t) \coloneqq y(t)$ (see e.g.\ \cite{Diethelm2010}).
Note that $D_{*}^{\alpha} y$ can also be defined for not necessarily
differentiable functions,
e.g.\ if $\alpha \in (0,1)$, for continuous functions $y$ for which 
$\lim_{t \to 0} t^{- \alpha} (v(t) - v(0))$ exists, is finite, and 
$
\lim_{\theta \uparrow 1} \sup_{t \in [0,T]} |\int_{\theta t}^{t} (t-s)^{-(\alpha + 1)} (v(t) - v(s)) \,ds| = 0
$, 
cf.\ \cite[Theorem 5.2]{Vainikko2016}.

For convenience, we use the notation
\begin{equation}
  D_*^{(\alpha_1, \ldots, \alpha_d)} x(t) 
  \coloneqq    
  \begin{pmatrix}
      D_{*}^{\alpha_1} 
   \\
      & \ddots
   \\
      & & D_{*}^{\alpha_d}
   \end{pmatrix}
   x(t)
   =
  \begin{pmatrix}
      D_{*}^{\alpha_1} x_1(t) 
   \\
      \vdots
   \\
      D_{*}^{\alpha_d} x_d(t)
   \end{pmatrix}.
\end{equation}
With $A = (a_{ij})_{i, j = 1, \dots, d} \in \C^{d \times d}$, $g = (g_1,\dots,g_d)^{\top}$, $x(t) = (x_1(t), \dots, x_d(t))^{\top}$, 
\eqref{eq:system} can then be rewritten as 
\begin{equation}\label{eq:system2}
   \begin{pmatrix}
      D_{*}^{\alpha_1} 
   \\
      & \ddots
   \\
      & & D_{*}^{\alpha_d}
   \end{pmatrix}
   x(t)
   =
   A x(t) + g(t).
\end{equation}
Of central importance are the two-parameter Mittag-Leffler functions $E_{\alpha, \beta} \colon \C \to \C$, $\alpha > 0$, $\beta \ge 0$, with
\begin{equation}
  \label{eq:ml2}
  E_{\alpha, \beta} (z) \coloneqq \sum_{j=0}^\infty \frac {z^j} {\Gamma(\alpha j + \beta)}
\end{equation}
and the one-parameter Mittag-Leffler functions $E_\alpha \colon \C \to \C$, $\alpha > 0$, defined by $E_{\alpha} \coloneqq E_{\alpha, 1}$ (see e.g.\ \cite{GorenfloEtal2014}).

The structure of the paper is as follows. In Section 2, we first introduce a
result on the existence and uniqueness of solutions to multi-order fractional
differential equations. Then, we give a representation of solutions to
homogeneous linear multi-order fractional differential equations in series
form. Section 3 is devoted to the study of the asymptotic behavior of solutions of linear multi-order
fractional differential equations. More precisely, we obtain some criterion on the asymptotic behavior of solutions to these equations. 
Some auxiliary results concerning the Mittag-Leffler functions 
and the asymptotic behavior of solutions of scalar linear fractional differential equations are shown in Appendix \ref{sec:appendix}.


\section{Fundamental theory of multi-order fractional differential equations}

In this section we provide some fundamental results regarding multi-order fractional differential equations.
Specifically, we shall prove a Picard-Lindel\"of type existence and uniqueness result in Subsection \ref{subs:existence-unique},
and Subsection \ref{subs:structure} will be devoted to a description of the structure of the associated solutions in the linear case.

\subsection{Existence and uniqueness of solutions to a class of multi-order fractional differential equations}
\label{subs:existence-unique}

Let $T > 0$. In this subsection we consider the existence and uniqueness of solutions to the multi-order fractional differential equation
\begin{equation}\label{eq_ex}
  D^{\alpha}_{*}x(t) = f(t, x(t)),\quad t\in (0,T],
\end{equation}
where $\alpha \coloneqq (\alpha_1,\dots,\alpha_d) \in (0,1]^d$ and
$f = (f_1,\dots,f_d)^{\top} \colon [0,T] \times \C^d \rightarrow \C^d$ is continuous. With similar arguments as in \cite[Chapter~6]{Diethelm2010} or
\cite[\S 3.5]{KilbasEtal2006}, one can show that for $x_0 = (x^0_1,\dots,x^0_d)^{\top} \in \C^d$ and a continuous function $x : [0, T] \to \C^d$ for which $D^{\alpha}_{*}x$ is defined (cf.\ \cite[Theorem 5.2]{Vainikko2016}), the following two statements are equivalent:
\begin{enumerate}
   \item[(i)] $x$ satisfies the $d$-dimensional differential equation system \eqref{eq_ex} together with the initial condition $x(0) = x_0$,
   \item[(ii)] $x$ satisfies the Volterra integral equation
      \begin{equation}\label{equi_eq}
         x(t) = x_0 + J^{\alpha} \big[f(\cdot, x(\cdot))\big] (t) \quad \forall t\in [0,T]
      \end{equation}
      where
      $J^{\alpha} [f(\cdot, x(\cdot))] (t) 
      \coloneqq \big(J^{\alpha_1} \big[f_1(\cdot, x(\cdot))\big] (t), 
      \dots, J^{\alpha_d} \big[f_d(\cdot, x(\cdot))\big] (t)\big)^{\top}$.
\end{enumerate}
Following the usual convention, we define solutions of \eqref{eq_ex} by considering \eqref{equi_eq}  for continuous functions.

\begin{definition}
  \label{def:solution}
  A continuous function $x : [0, T] \to \C^d$ is called a \emph{solution} to the differential
  equation \eqref{eq_ex} with initial condition $x(0) = x_0$
  if it satisfies the integral equation \eqref{equi_eq}.
\end{definition}

\begin{remark}
  Because we assume the function $f$ to be continuous, we can see that, for every solution $x$ of \eqref{eq_ex} 
  in the sense of Definition \ref{def:solution}, the function $f(\cdot, x(\cdot))$ is continuous, too. 
  Therefore, in view of the fact that the solution $x$ satisfies the integral equation \eqref{equi_eq}, 
  it follows for $i \in \{ 1, 2, \ldots, d \}$ that the $i$-th component of $x$ can be written as the sum of a constant and the 
  Riemann-Liouville integral of order $\alpha_i$ of a continuous function. 
  Using the arguments of \cite[proof of Theorem 3.7]{Diethelm2010}, we can then conclude that 
  $x_i$ fulfils the conditions of \cite[Theorem 5.2]{Vainikko2016} and thus that $D^{\alpha_i}_* x_i$ exists and is continuous.
  Therefore, under the continuity assumption on $f$, a solution to  \eqref{eq_ex} 
  in the sense of Definition \ref{def:solution} is automatically a strong solution to the differential equation in the classical sense.
\end{remark}

Our basic assumption on the given function $f$ will be that all its components 
$f_i : [0,T] \times \C^d\rightarrow \C$ are
continuous and satisfy a Lipschitz condition with respect to the second
variable, i.e.
\begin{equation}\label{Lip_cond}
  |f_i(t,y) - f_i(t,z)| \leq L \|y - z\| \quad \forall t \in [0,T], \; y, z \in \C^d
\end{equation}
with some constant $L>0$, where $\|\cdot\|$ is the max norm on $\C^d$, i.e.,
$\|y\|=\max \{ |y_1|, \dots, |y_d| \}$ for all
$y = (y_1, \dots, y_d)^{\top} \in \C^d$. 

We are now in a position to formulate a result on unique existence of solutions of initial value problems.

\begin{theorem}
  \label{thm:picard-lindeloef}
  Consider the equation \eqref{eq_ex}. Assume that the function $f$ is
  continuous and satisfies the Lipschitz condition \eqref{Lip_cond}. 
  Then, for any $x_0 = (x^0_1, \dots, x^0_d)^{\top} \in \C^d$, 
  the differential equation \eqref{eq_ex} has a unique solution in $C([0,T]; \C^d)$
  that satisfies the initial condition $x(0) = x_0$.
\end{theorem}

\begin{proof}
  Let $\lambda > 0$ be a constant such that 
  \begin{equation*}
    \max_{1\leq i\leq d} L \lambda^{-\alpha_i} < 1.
  \end{equation*}
  On the space $C([0,T];\C^d)$ we define a new norm $\|\cdot\|_\lambda$ as
  \[
    \| \varphi \|_\lambda \coloneqq \sup_{0\leq t\leq T} \| \varphi(t) \| \exp (-\lambda t) .
  \]
  Using standard arguments, it is easy to see that $(C([0,T];\C^d), \|\cdot\|_\lambda)$ is a Banach space. 
  For any $x_0 = (x^0_1, \dots, x^0_d)^{\top} \in \C^d$, 
  we define an operator $\mathcal{T}_{x_0} : C([0,T]; \C^d) \rightarrow C([0,T];\C^d)$ by
  \[
    \mathcal{T}_{x_0}\varphi(t) 
    \coloneqq \big((\mathcal{T}_{x_0}\varphi)^1(t),\dots,(\mathcal{T}_{x_0}\varphi)^d(t)\big)^{\top},
  \]
  where for $1 \leq i \leq d$
  \begin{equation*}\label{define_op}
    (\mathcal{T}_{x_0}\varphi)^i(t)
    \coloneqq x^0_i + \frac{1}{\Gamma(\alpha_i)} \int_0^t (t-\tau)^{\alpha_i-1} f_i(\tau,\varphi(\tau)) \;d \tau 
    \quad \forall t\in [0,T], \varphi\in  C([0,T];\C^d).
  \end{equation*}
  We see that for every $\varphi, \hat \varphi \in  C([0,T];\C^d)$, every $t \in [0,T]$ and all
  $1 \leq i \leq d$,
  \begin{align}
    \notag\frac{|(\mathcal{T}_{x_0}\varphi)^i(t)-(\mathcal{T}_{x_0}\hat{\varphi})^i(t)|}{\exp{(\lambda t)}}&\leq \frac{L}{\Gamma(\alpha_i)\exp{(\lambda t)}}\int_0^t (t-\tau)^{\alpha_i-1}\exp{(\lambda \tau)}\frac{\|\varphi(\tau)-\hat{\varphi}(\tau)\|}{\exp{(\lambda \tau)}}\;d\tau\\
    \notag&\leq \frac{L}{\Gamma(\alpha_i)\exp{(\lambda t)}}\int_0^t (t-\tau)^{\alpha_i-1}\exp{(\lambda \tau)}\;d\tau \cdot \sup_{0\leq \theta \leq T}\frac{\|\varphi(\theta)-\hat{\varphi}(\theta)\|}{\exp{(\lambda \theta)}}\\
    \notag&\leq \frac{L}{\Gamma(\alpha_i)}\int_0^t u^{\alpha_i-1}\exp{(-\lambda u)}\;du \cdot \| \varphi - \hat \varphi \|_\lambda \\
    \notag& =  \frac{L}{\Gamma(\alpha_i)\lambda^{\alpha_i}}\int_0^{\lambda t} v^{\alpha_i-1}\exp{(-v)}\;dv \cdot \| \varphi - \hat \varphi \|_\lambda \\
    \notag& \leq  \frac{L}{\Gamma(\alpha_i)\lambda^{\alpha_i}}\int_0^{\infty} v^{\alpha_i-1}\exp{(-v)}\;dv \cdot \| \varphi - \hat \varphi \|_\lambda \\
    &= \frac{L}{\lambda^{\alpha_i}} \|\varphi-\hat{\varphi}\|_\lambda . \label{es_eq}
  \end{align}
  It is clear that the operator $\mathcal{T}_{x_0}$ maps the space $(C([0,T];\C^d),\|\cdot\|_\lambda)$ to itself; moreover, from \eqref{es_eq} we obtain the estimate
  \[
    \| \mathcal{T}_{x_0} \varphi - \mathcal{T}_{x_0} \hat{\varphi} \|_\lambda
    \leq \frac{L}{\lambda^{\alpha_i}} \cdot \|\varphi-\hat{\varphi}\|_\lambda
    \quad \forall \varphi, \hat{\varphi} \in C([0,T];\C^d)
  \]
  which, by definition of $\lambda$, shows that this operator is a contractive mapping on this space. 
  Due to the fact that $(C([0,T];\C^d), \|\cdot\|_\lambda)$ is a Banach space, by Banach's fixed point theorem, 
  there exists a unique fixed point $\varphi$ of $\mathcal{T}_{x_0}$ in this space. 
  This fixed point is the unique solution of the Volterra equation \eqref{equi_eq} and hence, as stated above, 
  also the unique solution to the initial value problem consisting of the differential equation \eqref{eq_ex} 
  and the initial condition $x(0) = x_0$ in $C([0, T]; \C^d)$. The proof is complete.
\end{proof}

In Section \ref{sec:asymptotics} we shall look at the behavior of solutions to multi-order systems
as the independent variable goes to infinity. For this purpose, it is important to have an existence 
and uniqueness result that is not restricted to functions defined on bounded intervals.
Fortunately, the following result immediately follows from Theorem \ref{thm:picard-lindeloef}:

\begin{corollary}
  \label{cor:picard-lindeloef}
  Let $f : [0, \infty) \times \C^d \to \C^d$ be continuous and satisfy a Lipschitz condition with respect to the second variable.
  Moreover, let $\alpha \in (0,1]^d$ and $x_0 \in \C^d$.   
  Then, the initial value problem
  \[
    D_*^\alpha x(t) = f(t, x(t)) \quad (t > 0), \qquad x(0) = x_0,
  \] 
  has a unique solution in $C([0,\infty); \C^d)$.
\end{corollary}

\subsection{A representation of solutions to homogeneous linear multi-order fractional differential equations}
\label{subs:structure}

In this subsection we concentrate on a particularly important and fundamental special case of the class of differential equations discussed 
in Subsection \ref{subs:existence-unique}, namely we shall look at the solutions to homogeneous linear equations with constant coefficients, 
i.e.\ to differential equations of the form
\begin{equation}\label{eq:system2hom}
   \begin{pmatrix}
      D_{*}^{\alpha_1} 
   \\
      & \ddots
   \\
      & & D_{*}^{\alpha_d}
   \end{pmatrix}
   x(t)
   =
   A x(t)
\end{equation}
which is the special case of \eqref{eq:system2} where $g(t) = 0$ for all $t$.

Our basic result in this section, Theorem \ref{exist_Theorem}, provides some information about the structure of the solutions to the system \eqref{eq:system2hom} 
in the case of an arbitrary matrix $A \in \C^{d \times d}$ and an arbitrary vector $(\alpha_1, \ldots, \alpha_d) \in (0, 1]^d$.

In order to motivate our results, we start with the case $d = 2$. In this case, the system \eqref{eq:system2hom} has the form
\begin{subequations}
  \begin{align}
    D_*^{\alpha_1} x_1(t) & = a_{11} x_1(t) + a_{12} x_2(t), \\
    D_*^{\alpha_2} x_2(t) & = a_{21} x_1(t) + a_{22} x_2(t).
  \end{align}
\end{subequations}
First of all, Corollary \ref{cor:picard-lindeloef} asserts that, for any initial condition $(x_1(0), x_2(0))^\top = (x_1^0, x_2^0)^\top \in \C^2$, 
this system has a unique continuous solution $x = (x_1, x_2)^\top$ on $[0, \infty)$. 
Moreover, for equations of this structure, the fractional version of the variation-of-constants method 
\cite[Theorem~7.2 and Remark~7.1]{Diethelm2010} provides the relations
\begin{subequations}
  \label{eq:solution-general-2}
  \begin{align}
    x_1(t) 
    & = 
    x_1^0 E_{\alpha_1} (a_{11} t^{\alpha_1}) 
    + a_{12} \int_0^t (t-s)^{\alpha_1 - 1} E_{\alpha_1, \alpha_1} \left( a_{11} (t-s)^{\alpha_1} \right) x_2(s) \, ds, \\
    x_2(t) 
    & =  
    x_2^0 E_{\alpha_2} (a_{22} t^{\alpha_2}) 
    + a_{21} \int_0^t (t-s)^{\alpha_2 - 1} E_{\alpha_2, \alpha_2} \left( a_{22} (t-s)^{\alpha_2} \right) x_1(s) \, ds, 
  \end{align}
\end{subequations}
for all $t \ge 0$. 
This representation indicates that we should seek the solution components in the class of generalized power series of the form
\begin{subequations}
  \label{eq:powerseries}
  \begin{align}
    x_1(t) & =  x_1^0 + \sum_{k=1}^\infty \sum_{\ell=0}^\infty b_{1 k \ell} t^{k \alpha_1 + \ell \alpha_2}, \\
    x_2(t) & =  x_2^0 + \sum_{k=0}^\infty \sum_{\ell=1}^\infty b_{2 k \ell} t^{k \alpha_1 + \ell \alpha_2}.
  \end{align}
\end{subequations}
Assuming a suitable convergence behavior of these series, we may differentiate in a termwise manner and obtain
\begin{align*}
  D_*^{\alpha_1} x_1(t) 
  & = 
  \sum_{k=1}^\infty \sum_{\ell=0}^\infty 
       b_{1 k \ell} \frac {\Gamma(k \alpha_1 + \ell \alpha_2 + 1)} {\Gamma((k-1) \alpha_1 + \ell \alpha_2 + 1)} t^{(k-1) \alpha_1 + \ell \alpha_2} \\
  & =  
    \sum_{k=0}^\infty \sum_{\ell=0}^\infty 
       b_{1, k+1, \ell} \frac {\Gamma((k+1) \alpha_1 + \ell \alpha_2 + 1)} {\Gamma(k \alpha_1 + \ell \alpha_2 + 1)} t^{k \alpha_1 + \ell \alpha_2},
  \\
  D_*^{\alpha_2} x_2(t) 
  & =  
  \sum_{k=0}^\infty \sum_{\ell=1}^\infty 
       b_{2 k \ell} \frac {\Gamma(k \alpha_1 + \ell \alpha_2 + 1)} {\Gamma(k \alpha_1 + (\ell - 1) \alpha_2 + 1)} t^{k \alpha_1 + (\ell - 1) \alpha_2} \\
  & =  
    \sum_{k=0}^\infty \sum_{\ell=0}^\infty 
       b_{2, k, \ell+1} \frac {\Gamma(k \alpha_1 + (\ell + 1) \alpha_2 + 1)} {\Gamma(k \alpha_1 + \ell \alpha_2 + 1)} t^{k \alpha_1 + \ell \alpha_2}.
\end{align*}
Plugging these representations into the differential equation system \eqref{eq:system2hom}, we find
\begin{align*}
  \MoveEqLeft
  a_{11} x_1^0 + a_{11} \sum_{k=1}^\infty \sum_{\ell=0}^\infty b_{1 k \ell} t^{k \alpha_1 + \ell \alpha_2}  
   + a_{12} x_2^0 + a_{12} \sum_{k=0}^\infty \sum_{\ell=1}^\infty b_{2 k \ell} t^{k \alpha_1 + \ell \alpha_2}  \\
   = &  \sum_{k=0}^\infty \sum_{\ell=0}^\infty 
       b_{1, k+1, \ell} \frac {\Gamma((k+1) \alpha_1 + \ell \alpha_2 + 1)} {\Gamma(k \alpha_1 + \ell \alpha_2 + 1)} t^{k \alpha_1 + \ell \alpha_2}, \\
  \MoveEqLeft 
   a_{21} x_1^0 + a_{21} \sum_{k=1}^\infty \sum_{\ell=0}^\infty b_{1 k \ell} t^{k \alpha_1 + \ell \alpha_2}  
   + a_{22} x_2^0 + a_{22} \sum_{k=0}^\infty \sum_{\ell=1}^\infty b_{2 k \ell} t^{k \alpha_1 + \ell \alpha_2}  \\
   = &  \sum_{k=0}^\infty \sum_{\ell=0}^\infty 
       b_{2, k, \ell+1} \frac {\Gamma(k \alpha_1 + (\ell + 1) \alpha_2 + 1)} {\Gamma(k \alpha_1 + \ell \alpha_2 + 1)} t^{k \alpha_1 + \ell \alpha_2}.
\end{align*}
A comparison of coefficients of $t^{k \alpha_1 + \ell \alpha_2}$ then yields the equations
\begin{subequations}
  \label{eq:recurrence}
  \begin{align}
    b_{110} & =  \frac 1 {\Gamma(\alpha_1 + 1)} (a_{11} x_1^0 + a_{12} x_2^0), \label{eq:recurrence-a} \\
    b_{201} & =  \frac 1 {\Gamma(\alpha_2 + 1)} (a_{21} x_1^0 + a_{22} x_2^0), \label{eq:recurrence-b} \\
    b_{1,k+1,0} & = \frac {\Gamma(k \alpha_1 + 1)} {\Gamma((k+1) \alpha_1 + 1)} a_{11} b_{1k0} \qquad (k = 1, 2, \ldots), \label{eq:recurrence-c} \\
    b_{1,1,\ell} & = \frac {\Gamma(\ell \alpha_2 + 1)} {\Gamma(\alpha_1 + \ell \alpha_2 + 1)} a_{12} b_{20\ell} \qquad (\ell = 1, 2, \ldots), \label{eq:recurrence-d} \\
    b_{1,k+1,\ell} 
       & = 
       \frac {\Gamma(k \alpha_1 + \ell \alpha_2 + 1)} {\Gamma((k+1) \alpha_1 + \ell \alpha_2 + 1)} (a_{11} b_{1k\ell} + a_{12} b_{2k\ell})
       \qquad (k = 1, 2, \ldots; \ell = 1, 2, \ldots), \label{eq:recurrence-e} \\
    b_{2,0,\ell+1} & = \frac {\Gamma(\ell \alpha_2 + 1)} {\Gamma((\ell+1) \alpha_2 + 1)} a_{22} b_{20\ell} \qquad (\ell = 1, 2, \ldots), \label{eq:recurrence-f} \\
    b_{2,k,1} & = \frac {\Gamma(k \alpha_1 + 1)} {\Gamma(k \alpha_1 + \alpha_2 + 1)} a_{21} b_{1k0} \qquad (k = 1, 2, \ldots), \label{eq:recurrence-g} \\
    b_{2,k,\ell+1} 
       & = 
       \frac {\Gamma(k \alpha_1 + \ell \alpha_2 + 1)} {\Gamma(k \alpha_1 + (\ell + 1) \alpha_2 + 1)} (a_{21} b_{1k\ell} + a_{22} b_{2k\ell}) 
       \qquad (k = 1, 2, \ldots; \ell = 1, 2, \ldots). \label{eq:recurrence-h}
  \end{align}
\end{subequations}
Formally introducing the quantities
\begin{subequations}
  \label{eq:recurrence2}
  \begin{equation}
    \label{eq:recurrence2-a}
    \begin{aligned}
    b_{10\ell} & =  0     & \mbox{ for } \ell = 1, 2, \ldots \quad
    & \mbox{ and } & 
    b_{2k0} & =  0 & \mbox{ for } k = 1,2,\ldots, & \\
    b_{100} & =     x_1^0 & 
    & \mbox{ and } & 
    b_{200} & =  x_2^0, & 
    \end{aligned}
  \end{equation}
  we see that the system \eqref{eq:recurrence} can be simplified to
  \begin{align}
    b_{1,k+1,\ell} 
       &=
       \frac {\Gamma(k \alpha_1 + \ell \alpha_2 + 1)} {\Gamma((k+1) \alpha_1 + \ell \alpha_2 + 1)} (a_{11} b_{1k\ell} + a_{12} b_{2k\ell})
       \qquad (k, \ell = 0, 1, 2, \ldots) , \label{eq:recurrence2-c} \\
    b_{2,k,\ell+1} 
       &= 
       \frac {\Gamma(k \alpha_1 + \ell \alpha_2 + 1)} {\Gamma(k \alpha_1 + (\ell + 1) \alpha_2 + 1)} (a_{21} b_{1k\ell} + a_{22} b_{2k\ell}) 
       \qquad (k, \ell = 0, 1, 2, \ldots). \label{eq:recurrence2-d}
  \end{align}
\end{subequations}
A brief inspection of these formulas reveals that, given the initial values $x_1^0$ and $x_2^0$, 
they can indeed be used to compute all coefficients that appear in the representation \eqref{eq:powerseries} in a recursive manner. 
Specifically, the coefficients $b_{1k\ell}$ and $b_{2k\ell}$ for $k + \ell = \mu$ can be computed via eqs.\ \eqref{eq:recurrence2-c} and \eqref{eq:recurrence2-d}, 
respectively, and this computation only requires the knowledge of $b_{1k\ell}$ and $b_{2k\ell}$ with $k + \ell = \mu - 1$. 
Thus one can first compute all $b_{1k\ell}$ and $b_{2k\ell}$ with $k + \ell = 1$, then with $k + \ell = 2$, etc.

A closer look at the recurrence relations \eqref{eq:recurrence2} allows us to prove that the series from \eqref{eq:powerseries} converge for all $t \ge 0$.
To this end we first state a preliminary result.

\begin{lemma}
  \label{lemma:general-system-series-convergence}
  Let the values $b_{1 k \ell}$ and $b_{2 k \ell}$ be defined as in \eqref{eq:recurrence2} with arbitrary $x_1^0, x_2^0 \in \C$.
  Then, for $j \in \{ 1, 2 \}$ the series 
  \begin{equation*}
    s_j(z) \coloneqq \sum_{k=0}^\infty \sum_{\ell=0}^\infty |b_{j k \ell}| z^{k + \ell}
  \end{equation*}
  is convergent for all $z \in \C$.
\end{lemma}

Actually it is immediately clear that the desired convergence property is a consequence of this lemma, since the series
\begin{equation*}
  \sum_{k=0}^\infty \sum_{\ell=0}^\infty |b_{j k \ell}| \cdot |t|^{k \alpha_1 + \ell \alpha_2}
\end{equation*}
is, on the one hand, a majorant for $x_j(t)$ and is, on the other hand, convergent for all $t > 0$ according to
\begin{equation*}
  \sum_{k=0}^\infty \sum_{\ell=0}^\infty |b_{j k \ell}| \cdot |t|^{k \alpha_1 + \ell \alpha_2}
  \le
     \begin{cases}
       \displaystyle
       \sum_{k=0}^\infty \sum_{\ell=0}^\infty |b_{j k \ell}| \cdot |t|^{(k+\ell) \max \{\alpha_1, \alpha_2\}}
           = s_j(t^{\max \{\alpha_1, \alpha_2\}}) & \mbox{for } t \ge 1, \smallskip \\
       \displaystyle
       \sum_{k=0}^\infty \sum_{\ell=0}^\infty |b_{j k \ell}| \cdot |t|^{(k+\ell) \min \{\alpha_1, \alpha_2\}}
           = s_j(t^{\min \{\alpha_1, \alpha_2\}}) & \mbox{for } t < 1.
     \end{cases}
\end{equation*}

\begin{proof}[Proof of Lemma \ref{lemma:general-system-series-convergence}]
  Since the series in question does not have any negative summands, we may rearrange the terms according to powers of $z$; this yields
  \begin{equation*}
    s_j(z) = \sum_{k=0}^\infty \sum_{\mu=0}^k |b_{j, \mu, k-\mu}| z^k.
  \end{equation*}
  It is therefore evident that, in order to investigate the convergence radius of this series, we need to estimate
  expressions of the form 
  \begin{equation*}
    \beta_{jk} \coloneqq \sum_{\mu=0}^k |b_{j, \mu, k-\mu}|.
  \end{equation*}
  In fact, we shall demonstrate that for sufficiently large $k$
  \begin{equation}
    \label{eq:bound-beta}
    0 \le \beta_{1k} + \beta_{2k} \le \frac {c_1 c_2^k} {\Gamma (k \alpha^* + 1)},
  \end{equation}
  where $c_1$ and $c_2$ are certain positive constants and
  \begin{equation*}
    \alpha^* \coloneqq \min \{ \alpha_1, \alpha_2 \}.
  \end{equation*}
  Equation \eqref{eq:bound-beta} tells us that the classical power series for the Mittag-Leffler function $E_{\alpha^*}$
  --- that is well known to be convergent on the entire complex plane --- evaluated at $c_2 |z|$ is a majorant for the series $s_1$ and $s_2$ that we
  are interested in, and hence the series expansions for $s_1(z)$ and $s_2(z)$ also converge for all $z$ as required.

  Thus, it only remains to prove \eqref{eq:bound-beta}. 
  The left inequality is clear by definition.
  To prove the right inequality, we employ the relations \eqref{eq:recurrence2-a}, \eqref{eq:recurrence2-c} 
  and \eqref{eq:recurrence2-d} and see,
  using the notation $\bar a := \max_{i,j \in \{ 1, 2 \}} |a_{ij}|$, that we have for $k \ge 2$ the following chain of inequalities:
  \begin{align*}
    \sum_{\mu=0}^k \left( |b_{1, \mu, k-\mu}|  + |b_{2, \mu, k-\mu}|  \right) 
    & \le   \sum_{\mu=1}^k  |b_{1, \mu, k-\mu}|  + \sum_{\mu=0}^{k-1} |b_{2, \mu, k-\mu}|  \\
    & \le  
       \bar a \sum_{\mu=1}^k \frac {\Gamma((\mu-1) \alpha_1 + (k-\mu) \alpha_2 + 1)} {\Gamma(\mu \alpha_1 + (k-\mu) \alpha_2 + 1)} 
                                         (|b_{1, \mu-1, k-\mu}|  + |b_{2, \mu-1, k-\mu}| ) \\
    & \phantom{\le} {} + \bar a \sum_{\mu=0}^{k-1} \frac {\Gamma(\mu \alpha_1 + (k-\mu-1) \alpha_2 + 1)} {\Gamma(\mu \alpha_1 + (k-\mu) \alpha_2 + 1)} 
                                         (|b_{1, \mu, k-\mu-1}|  + |b_{2, \mu, k-\mu-1}| ) \\
    & =  
       \bar a \sum_{\mu=1}^k \frac {\Gamma((\mu-1) \alpha_1 + (k-\mu) \alpha_2 + 1)} {\Gamma(\mu \alpha_1 + (k-\mu) \alpha_2 + 1)} 
                                         (|b_{1, \mu-1, k-\mu}|  + |b_{2, \mu-1, k-\mu}| ) \\
    & \phantom{\le} {} + \bar a \sum_{\mu=1}^{k} \frac {\Gamma((\mu-1) \alpha_1 + (k-\mu) \alpha_2 + 1)} {\Gamma((\mu-1) \alpha_1 + (k-\mu+1) \alpha_2 + 1)} 
                                         (|b_{1, \mu-1, k-\mu}|  + |b_{2, \mu-1, k-\mu}| ) \\
    & = \bar a \sum_{\mu=1}^k w_{k,\mu}(\alpha_1, \alpha_2)  (|b_{1, \mu-1, k-\mu}|  + |b_{2, \mu-1, k-\mu}| )
  \end{align*}
  with
  \begin{equation*}
     w_{k,\mu}(\alpha_1, \alpha_2) 
     = \frac {\Gamma((\mu-1) \alpha_1 + (k-\mu) \alpha_2 + 1)} {\Gamma(\mu \alpha_1 + (k-\mu) \alpha_2 + 1)} +
            \frac {\Gamma((\mu-1) \alpha_1 + (k-\mu) \alpha_2 + 1)} {\Gamma((\mu-1) \alpha_1 + (k-\mu+1) \alpha_2 + 1)}.
  \end{equation*}
  Both fractions on the right-hand side have the same numerator but their denominators differ by $\alpha_2 - \alpha_1$; 
  the well known monotonicity of the Gamma function thus allows us to conclude that, for sufficiently large $k$, we have
  \begin{equation}
    \label{eq:gamma1}
     w_{k,\mu}(\alpha_1, \alpha_2) 
     \le 2 \frac {\Gamma((\mu-1) \alpha_1 + (k-\mu) \alpha_2 + 1)} {\Gamma((\mu-1) \alpha_1 + (k-\mu) \alpha_2 + \alpha^* + 1)} 
     =   2 \frac {\Gamma(u + \mu (\alpha_1 - \alpha_2))} {\Gamma (u + \mu (\alpha_1 - \alpha_2) + \alpha^*)}
  \end{equation}
  with $u \coloneqq - \alpha_1 + k \alpha_2 + 1$. 
  For $\gamma > 0$ and $z \to \infty$, 
  Stirling's formula yields the asymptotic relation $\Gamma(z) / \Gamma(z + \gamma) = z^{-\gamma} (1 + o(1))$ which is monotonically decreasing in $z$.
  Hence, for sufficiently large $k$, the quotient on the right-hand side of \eqref{eq:gamma1} is monotonically decreasing with respect to $\mu$ for
  $\alpha_1 \ge \alpha_2$ and monotonically increasing with respect to $\mu$ if $\alpha_1 < \alpha_2$. 
  Therefore, the maximum of this expression over all admissible values of $\mu$ is attained at $\mu = 1$ if $\alpha_1 \ge \alpha_2$ and at 
  $\mu = k$ if $\alpha_1 < \alpha_2$. These observations may be summarized in the form
  \begin{equation*}
     w_{k,\mu}(\alpha_1, \alpha_2) 
     \le
     2 \frac {\Gamma( \frac{k-1}2 ( \alpha_1 + \alpha_2 - | \alpha_1 - \alpha_2|) + 1 )} {\Gamma(\frac{k-1}2 ( \alpha_1 + \alpha_2 - | \alpha_1 - \alpha_2| ) + \alpha^* + 1)}
     = 
     2 \frac {\Gamma((k-1) \alpha^* + 1)} {\Gamma(k \alpha^* + 1)},
  \end{equation*}
  and this implies
  \begin{align*}
    \beta_{1k} + \beta_{2k}
    & \le 
    2 \bar a \frac {\Gamma((k-1) \alpha^* + 1)} {\Gamma(k \alpha^* + 1)} \sum_{\mu=1}^k (|b_{1, \mu-1, k-\mu}|  + |b_{2, \mu-1, k-\mu}| ) \\
    & = 2 \bar a \frac {\Gamma((k-1) \alpha^* + 1)} {\Gamma(k \alpha^* + 1)} \sum_{\mu=0}^{k-1} (|b_{1, \mu, k-1-\mu}|  + |b_{2, \mu, k-1-\mu}| ) \\
    & = 2 \bar a \frac {\Gamma((k-1) \alpha^* + 1)} {\Gamma(k \alpha^* + 1)} ( \beta_{1,k-1} + \beta_{2,k-1}  )
  \end{align*}
  if $k$ is large enough. Thus, for a sufficiently large and fixed constant $N$ and arbitrary $k$, by induction, we deduce the estimate
  \[
    \beta_{1,N+k}+\beta_{2,N+k}\leq \frac{(2 \bar{a})^k \Gamma(N\alpha^*+1)}{\Gamma((N+k)\alpha^*+1)}(\beta_{1,N}+\beta_{2,N}),
  \]
  which shows \eqref{eq:bound-beta} and completes the proof of Lemma \ref{lemma:general-system-series-convergence}.
\end{proof}

The same ideas and methods can be applied if the dimension of the fractional differential equation system is greater than 2.
We summarize the findings as follows.

\begin{theorem}\label{exist_Theorem}
  Let $\alpha = (\alpha_1, \ldots, \alpha_d) \in (0, 1]^d$ and $A \in \C^{d \times d}$.
  Then, for each $x_0 \in \C^d$, the initial value problem
  \begin{equation}
    \label{eq:ivp}
    D_*^\alpha x(t) = A x(t), \qquad x(0) = x_0,
  \end{equation}
  has a uniquely determined solution in $C([0,\infty); \C^d)$. 
  The components of this solution can be expressed in the form 
  \begin{equation}
    \label{eq:series}
    x_j(t) 
    = 
    \sum_{k=0}^\infty \,\, \sum_{\ell_1, \ell_2, \ldots, \ell_{j-1}, \ell_{j+1}, \ldots, \ell_d = 1}^\infty 
              b_{k, \ell_1, \ell_2, \ldots, \ell_{j-1}, \ell_{j+1}, \ldots, \ell_d} t^{k \alpha_j + \sum_{\mu=1, \mu \ne j}^d \ell_\mu \alpha_\mu},
  \end{equation}
  and the series in eq.\ \eqref{eq:series} converges for all $t \ge 0$.
\end{theorem}

\section{Asymptotic behavior of solutions of multi-order fractional differential equations}
\label{sec:asymptotics}

Having established these foundations, we can now come to the core of this paper, 
namely the discussion of the asymptotic behavior of solutions of linear multi-order fractional differential systems.
\subsection{Systems with (block) triangular coefficient matrices}
\label{subs:triangular}

Assume that $\alpha_k \in (0,1]$ for $1 \leq k \leq d$. 
In the case that the coefficient matrix $A$ of the system \eqref{eq:system} has a triangular structure, 
we provide a detailed investigation of the asymptotic behaviour of the system's solutions. 
More precisely, we obtain a necessary and sufficient condition such that all solutions of the homogeneous system
associated to \eqref{eq:system} tend to zero at infinity, and we derive sufficient conditions for all solutions of the
full inhomogeneous system \eqref{eq:system} to have this property.
In this context we stress that the $\alpha_k$ may be completely arbitrary numbers from the interval $(0,1]$; 
in particular it is allowed that $\alpha_k = \alpha_{k'}$ for some $k \ne k'$.

Thus, let us now first consider the system
\begin{subequations}
  \begin{equation}\label{upperEq}
    D_*^{\alpha_i} x_i(t) = \sum_{j=i}^{d} a_{ij} x_j(t), \qquad 1 \leq i \leq d,
  \end{equation}
  i.e.\ the case of a homogeneous system with an upper triangular matrix $A$, 
  together with the initial condition
  \begin{equation}
    x_i(0) = x_i^0, \qquad 1 \le i \le d.
  \end{equation}
\end{subequations}
In order to exclude the pathological and practically irrelevant case where
the right-hand sides of certain equations from the system \eqref{upperEq} do not depend on their respective unknown functions, 
we shall explicitly assume throughout this subsection that $a_{ii} \ne 0$ for all $i = 1,2,\ldots, d$.
In other words, we assume the matrix $A$ to be not only upper triangular but also nonsingular.

The case where $A$ is of lower triangular form can be handled in a completely analog manner;
we shall not treat this case explicitly.
The associated inhomogeneous system will be discussed later; cf.\ Corollary \ref{cor:inhomogeneous}.

In the simplest nontrivial case $d = 2$,
the system \eqref{upperEq} has the form
\begin{flalign*}
  D_*^{\alpha_1} x_1(t) 
   &=
   a_{11} x_1(t) + a_{12} x_2(t),
   \\
   D_*^{\alpha_2} x_2(t) 
   &=
   \phantom{a_{11} x_1(t) + {}} a_{22} x_2(t),
\end{flalign*}
and it is a relatively simple matter to explicitly compute its solution. Specifically, in view of the triangular structure of the coefficient matrix,
one can solve the second equation of the system directly and obtain the well known result \cite[Theorem~4.3]{Diethelm2010}
\begin{subequations}
  \label{eq:solutions-2d-triangular}
  \begin{equation}
    x_2(t) 
    = 
    x_2^0 E_{\alpha_2} (a_{22} t^{\alpha_2}).
  \end{equation}
  This result can be plugged into the system's first equation which then takes the form
  \begin{equation*}
     D_*^{\alpha_1} x_1(t) 
     =
     a_{11} x_1(t) + a_{12} x_2^0 E_{\alpha_2} (a_{22} t^{\alpha_2}).
  \end{equation*}
  For equations of this structure, the fractional version of the variation-of-constants method \cite[Theorem~7.2 and Remark~7.1]{Diethelm2010} provides the solution
  \begin{equation}
    x_1(t) 
    = 
    x_1^0 E_{\alpha_1} (a_{11} t^{\alpha_1}) 
      + a_{12} x_2^0 \int_0^t (t-s)^{\alpha_1 - 1} E_{\alpha_1, \alpha_1} \left( a_{11} (t-s)^{\alpha_1} \right) E_{\alpha_2} (a_{22} s^{\alpha_2}) \, ds.
  \end{equation}
\end{subequations}
From the representation (\ref{eq:solutions-2d-triangular}) it is evident 
that the solution vector $(x_1, x_2)^\top$ is an element of the function space $C[0,\infty)$. 
Moreover, the power series representations of the Mittag-Leffler functions imply that the component $x_2(t)$ can be written as a power series in $t^{\alpha_2}$, 
and therefore its asymptotic behavior as $t \to 0+$ is of the form 
\begin{equation*}
  x_2(t) 
  = 
  x_2^0 + \frac {c_2 a_{22}} {\Gamma(\alpha_2 + 1)} t^{\alpha_2} + O(t^{2 \alpha_2}),
\end{equation*}
whereas the behavior of $x_1(t)$ in this respect can be described by
\begin{equation*}
  x_1(t) 
  = 
  x_1^0 + \frac {c_1 a_{11}} {\Gamma(\alpha_1 + 1)} t^{\alpha_1} + o(t^{\alpha_1})
\end{equation*}
with some constant $c_1 \in \C$.
The arguments employed in Subsection \ref{subs:structure} can be used to derive more details.

These considerations can directly be generalized to homogeneous upper triangular systems of arbitrary dimension $d$. 
In this case we obtain the set of equations
\begin{equation}
  \label{eq:solution}
  x_i(t) 
  = 
  x_i^0 E_{\alpha_i} (a_{ii} t^{\alpha_i}) 
      + \sum_{j=i+1}^d a_{ij} \int_0^t (t-s)^{\alpha_i - 1} E_{\alpha_i, \alpha_i} \left( a_{ii} (t-s)^{\alpha_i} \right) x_j(s) \, ds
\end{equation}
for $i = d, d-1, \ldots, 1$ which can be recursively evaluated to explicitly compute the solutions.

Some known results about the asymptotic behavior of the Mittag-Leffler functions admit to draw the conclusions required in the asymptotic behavior analysis. 
The main result in this context is the following theorem. 
The proof of its statements requires a number of auxiliary results that can be considered as minor extensions of already known theorems and lemmas.
Since these extensions may be of a certain degree of independent interest, we have explicitly formulated and collected them, together with
complete proofs, in Appendix \ref{sec:appendix}.

\begin{theorem}\label{thm:characterSpec}
  \begin{itemize}
  \item[(i)] Every solution of the system \eqref{upperEq} converges to zero at infinity if and only if 
    \begin{equation}\label{eq:characterSpec}
      |\arg{(a_{kk})}| > \frac {\alpha_k \pi} {2} \qquad \forall k \in \{ 1, \dots, d \}.
    \end{equation}
  \item[(ii)] If there exists $k \in \{ 1, \dots, d \}$ such that
    $|\arg{(a_{kk})}| < \alpha_k \pi / 2$ then there exists some $x_0$ such
    that the solution to the system \eqref{upperEq} that satisfies the initial
    condition $x(0) = x_0$ is unbounded.
  \end{itemize}
\end{theorem}

\begin{proof}
  For the proof of part (i),
  we will first show that the condition \eqref{eq:characterSpec} is sufficient to assert that every solution of \eqref{upperEq} converges to zero at infinity.
  Indeed, for any initial value $x_0 = (x_1^0, \dots, x^0_d)^{\top}\in \C^d$, 
  we denote the solution of \eqref{upperEq} starting from $x_0$ by $\varphi (\cdot, x_0) = (\varphi_1(\cdot, x_0), \ldots, \varphi_d(\cdot, x_0))^\top$.
  Our proof will use mathematical induction over the index $j$ of the components of the solution vector in a backward direction. 
  Thus, for our induction basis we consider $j = d$.
  Since the $d$-th equation of the system \eqref{upperEq} reads
  \[
    D_*^{\alpha_d} x_d(t) = a_{dd} x_d(t),
  \]
  it follows from Lemma \ref{stableLemma}(i) that
  the condition $|\arg a_{dd}| > \alpha_d \pi / 2$ is sufficient to assert that $\varphi_d (t, x_0) \to 0$ as $t \to \infty$ for all $x_0$.
  For the induction step, we assume that we have already shown that the components $d$, $d-1$, \ldots, $j+1$ of the solution tend to $0$ as $t \to \infty$
  for any choice of the initial values. Then we need to prove that this is also true for the $j$-th component. 
  To this end we recall that the $j$-th component of the differential equation system \eqref{upperEq} reads
  \begin{equation*}
    D_*^{\alpha_j} x_j(t) = a_{jj} x_j(t) + \sum_{k=j+1}^d a_{jk} \varphi_k(t, x_0).
  \end{equation*}
  All the terms in the sum are already known and, because of the induction hypothesis, they are continuous and tend to zero as $t \to \infty$.
  Thus we may apply Lemma \ref{stableLemma}(i) and immediately deduce that $x_j$ has this property as well.
  
  To conclude the proof of part (i) we now have to demonstrate that \eqref{eq:characterSpec} is also necessary for all solutions of \eqref{upperEq} tend to zero as $t\to\infty$. To this end we assume that \eqref{eq:characterSpec} does not hold.
  Then there exists an index $k_0 \in \{ 1, 2, \ldots, d \}$ which satisfies 
  \begin{equation*}
    |\arg{(a_{ii})}| > \frac{\alpha_i \pi}{2} \mbox{ for }  k_0 + 1 \leq i \leq d 
    \quad \mbox{ and } \quad
    |\arg{(a_{k_0,k_0})}| \leq \frac{\alpha_{k_0} \pi}{2},
  \end{equation*}
  i.e.\ $k_0$ is the largest index for which \eqref{eq:characterSpec} is violated.
  Consider the equation
  \begin{equation}\label{tam}
    D^{\alpha_{k_0}}_* x_{k_0}(t) = a_{k_0k_0} x_{k_0}(t) + f(t) 
    \qquad \mbox{ with } \qquad
    f(t) \coloneqq \sum_{i = k_0 + 1}^d a_{k_0i} x_i(t).
  \end{equation}
  Since \eqref{eq:characterSpec} is true for all $i > k_0$, the arguments used
  above imply that $f$ is continuous and tends to zero at infinity. 
  As in the considerations above, we may use the fractional variation-of-constants method \cite[Theorem~7.2 and Remark~7.1]{Diethelm2010}
  to see that the set of all solutions to \eqref{tam} consists of the functions 
  \begin{equation}
    \label{eq:solution-tam}
    \varphi_{k_0}(t, x_0) =  x_{k_0}^0 E_{\alpha_{k_0}} (a_{k_0,k_0} t^{\alpha_{k_0}}) + h(t)
  \end{equation}
  where $x_{k_0}^0$ runs through the entire complex plane and where
  \[
    h(t) 
    \coloneqq \int_0^t (t - \tau)^{\alpha_{k_0} - 1} E_{\alpha_{k_0}, \alpha_{k_0} } (a_{k_0,k_0} (t-\tau)^{\alpha_{k_0}}) f(\tau) \; d\tau.
  \]
  The well known asymptotic behavior of the Mittag-Leffler functions \cite[Proposition~3.6 and Theorem~4.3]{GorenfloEtal2014} then implies
  that $E_{\alpha_{k_0}} (a_{k_0,k_0} t^{\alpha_{k_0}})$ does not converge to $0$ as $t \to \infty$
  because of our assumption on the relation of $\alpha_{k_0}$ and $| \arg a_{k_0, k_0}|$.
  Now assume that there exists some $x_{k_0}^0 \in \C$ such that $\varphi_{k_0}(t, x_0) \to 0$ as $t \to \infty$. 
  Then, it follows that for every $\tilde x_0 \in C^d$ with $x_k^0 = \tilde x_k^0$ for $k = k_0+1, \ldots, d$ and
  $x_{k_0}^0 \neq \tilde x_{k_0}^0$, we have
  \begin{align*}
    \varphi_{k_0}(t, \tilde x_0) 
    & = \tilde x_{k_0}^0 E_{\alpha_{k_0}} (a_{k_0,k_0} t^{\alpha_{k_0}}) + h(t) \\
    & = (\tilde x_{k_0}^0 - x_{k_0}^0) E_{\alpha_{k_0}} (a_{k_0,k_0} t^{\alpha_{k_0}}) + \varphi_{k_0}(t, x_0).
  \end{align*}
  For $t \to \infty$, the last summand on the right-hand side of this equality tends to zero but the other summand does not,
  and hence we conclude that $\varphi(t, \tilde x_0)$ does not tend to zero as $t \to \infty$ which yields our required contradiction.

  For the proof of part (ii), we --- much as above --- know that
  there exists an index $k_0 \in \{ 1, 2, \ldots, d \}$ which satisfies 
  \begin{equation*}
    |\arg{(a_{ii})}| \ge \frac{\alpha_i \pi}{2} \mbox{ for }  k_0 + 1 \leq i \leq d 
    \quad \mbox{ and } \quad
    |\arg{(a_{k_0k_0})}| < \frac{\alpha_{k_0} \pi}{2}.
  \end{equation*}
  We may then proceed in the same way as in the second part of the proof of (i).
  However, now we know that $|E_{\alpha_{k_0}} (a_{k_0k_0} t^{\alpha_{k_0}})| \to \infty$ as $t \to \infty$,
  and therefore we may even conclude that there exists some $x_{k_0}^0 \in \C$ such that $\varphi_{k_0}(t, x_0)$ is unbounded.
\end{proof}

\begin{remark}
  \label{rmk:block-triangular}
  The same arguments can be used if the coefficient matrix $A$ of the system has a block-upper triangular structure and the 
  differentiation matrix on the left-hand side of the differential equation has a block structure with identical block sizes
  where each block consists of differential operators of the same order, i.e.\ if the differential equation has the form
  \begin{equation}\label{eq:system3}
    \begin{pmatrix}
      D_1
      \\
      & \ddots
      \\
      & & D_n
    \end{pmatrix}
    x(t)
    =
    \begin{pmatrix} 
      A_{11} & A_{12} & \cdots & A_{1n} 
      \\
             & A_{22} &        & A_{2n}
      \\
             &        & \ddots & \vdots
      \\
             &        &        & A_{nn}
    \end{pmatrix}
    x(t)
  \end{equation}
  where, using the notation $I_\mu$ for the $\mu$-dimensional unit matrix,
  \begin{equation*}
     D_j = D_*^{\alpha_j} I_{d_j},
  \end{equation*}
  $A_{jk} \in \C^{d_j \times d_k}$ and $x = (x_1, \ldots, x_d)^\top$ with $d = \sum_{j=1}^n d_j$: 
  In this case, 
  \begin{itemize}
  \item all solutions of the system \eqref{eq:system3} converge to zero as $t \to \infty$ if and only if, for all $j=1,2,\ldots,n$, all 
    eigenvalues $\lambda_{jk}$, $k=1,2,\ldots, d_j$, of the matrix $A_{jj}$ satisfy
    $|\arg \lambda_{jk}| > \alpha_j \pi /2$, and 
  \item whenever there
    exist some $j \in \{ 1, 2, \ldots, n \}$ and $k \in \{ 1, 2, \ldots, d_j \}$ 
    with $|\arg \lambda_{jk}| < \alpha_j \pi /2$, there exists an initial value whose
    corresponding solution is unbounded.
  \end{itemize}
\end{remark}

A close inspection of the proof of Theorem \ref{thm:characterSpec} reveals 
that the statement of its part (i) can easily be extended to cover a class of inhomogeneous problems:

\begin{corollary}
  \label{cor:inhomogeneous}
  Consider the differential equation system
  \begin{equation}\label{inhomogeneous}
    D_*^{\alpha_i} x_i(t) = \sum_{j=i}^{d} a_{ij} x_j(t) + g_i(t), \qquad 1 \leq i \leq d,
  \end{equation}
  where, for all $i = 1,2, \ldots, d$, the functions $g_i : [0, \infty) \rightarrow \C$ are continuous 
  and satisfy
  \[ 
    \lim_{t \to \infty} g_i(t) = 0.
  \]
Every solution of the inhomogeneous system \eqref{inhomogeneous} converges to zero at infinity if and only if all solutions of the associated homogeneous system \eqref{upperEq} tend to zero as $t\to\infty$, i.e.\ if and only if condition \eqref{eq:characterSpec} is satisfied.
\end{corollary}

\begin{proof}
  Assume first that every solution of \eqref{inhomogeneous} tends to zero as $t \to \infty$. 
  In order to prove that every solution of the corresponding homogeneous system \eqref{upperEq} converges to zero, we choose an arbitrary $x_0 \in \C^d$.
  It is then sufficient to show that the solution of \eqref{upperEq} that starts at $x_0$ converges to zero as $t \to \infty$. 
  To this end, we take the solutions $\varphi(\cdot,x_0)$ and $\varphi(\cdot,0)$ of \eqref{inhomogeneous} that start at $x_0$ and at $0$, respectively.
  By assumption, both these functions tend to $0$ as $t \to \infty$.
  Thus, $\varphi(t, x_0) - \varphi(t, 0)$ tends to $0$ as $t \to \infty$ as well.
  But clearly, this difference is identical to the solution of the homogeneous system \eqref{upperEq} that starts at $x_0$.

  Regarding the proof of the other direction of the equivalence, we assume that the condition \eqref{eq:characterSpec} is satisfied. Under this hypothesis, we may proceed as in the first part of the proof of Theorem \ref{thm:characterSpec}(i).
  Using the argumentation via Lemma \ref{stableLemma}(i) employed in the induction step there, 
  we can derive that $\varphi_d(t, x_0) \to 0$ as $t \to \infty$ for any $x_0 \in \C^d$.
  Then we can proceed inductively as in the first part of the proof of Theorem \ref{thm:characterSpec}(i)
  and demonstrate that the other components of $\varphi(\cdot, x_0)$ vanish near infinity as well. The proof is complete.
\end{proof}

\begin{remark}
  Clearly, the same arguments can be used to extend the statement of Remark \ref{rmk:block-triangular} regarding block triangular systems 
  to the inhomogeneous case as well.
\end{remark}

\subsection{Systems with general coefficient matrices}
\label{subs:general-orders}

With respect to the stability theory for such systems of equations with general (not necessarily triangular or block triangular) coefficient matrices, 
we are not yet in a position to provide a comprehensive theory. 
We can, however, develop an approach that works under certain restrictions on the orders of the differential 
operators involved. 
Specifically we shall assume that $\alpha_j \in (0,1]$ for all $j$
and that there exists some $\alpha^* \in (0,1]$ and some $\rho_j \in \Q$ such that $\alpha_j = \rho_j \alpha^*$.

In this case, there exist positive integers $p_j$ and $q_j$ ($j=1,2,\ldots,d$) such that, for all $j$, 
$\gcd (p_j, q_j) = 1$ and $\rho_j = p_j / q_j$. Then we define $q$ to be the least common multiple of the $q_j$.
This allows us to deduce that for every $j$ there exists some positive integer $r_j$ 
such that $\alpha_j = \alpha^* r_j / q$ (clearly, $r_j = p_j q / q_j$).
According to \cite[Theorem~8.1]{Diethelm2010}, we can then rewrite the $j$-th equation of the original system \eqref{eq:system} 
as an equivalent system of $r_j$ differential equations of order $\alpha^* / q$. 
Thus, the entire system \eqref{eq:system} can be expressed as a system of $d^* = \sum_{j=1}^d r_j$ equations of order $\alpha^* / q$.
This new system has the form
\begin{subequations}
  \begin{equation}\label{eq:commensurate-matrix}
    D_*^{\alpha^* / q} x^*(t) = A^* x^*(t) + g^*(t)
  \end{equation}
  where the matrix $A^*$ has the block structure
  \begin{equation}
    \label{eq:commensurate-matrix-a}
    A^* = 
    \begin{pmatrix*}
      A_{11} & A_{12} & \cdots & A_{1d} 
      \\
      A_{21} & A_{22} & \cdots & A_{2d} 
      \\
      \vdots &       & \ddots & \vdots 
      \\
      A_{d1} & A_{d2} & \cdots & A_{dd}      
    \end{pmatrix*}
  \end{equation}
  with matrices $A_{jk} \in \C^{r_j \times r_k}$ given by
  \begin{equation}
    \label{eq:commensurate-matrix-b}
    A_{jj} = 
    \begin{pmatrix*}
      0 & 1 & 0 & \cdots & 0 \\
      0 & 0 & 1 & \ddots & \vdots \\
      \vdots & & \ddots & \ddots & 0 \\
      0 & 0& \cdots & 0 & 1 \\
      a_{jj} & 0 & \cdots & 0 & 0
    \end{pmatrix*}
    \quad \mbox{ for } j = 1,2, \ldots, d
  \end{equation}
  and
  \begin{equation}
    \label{eq:commensurate-matrix-c}
    A_{jk} = 
    \begin{pmatrix*}
      0 & 0 &  \cdots & 0 \\
      \vdots & \vdots & \ddots & 0 \\
      0 & 0 & \cdots & 0 \\
      a_{jk} & 0 & \cdots & 0 
    \end{pmatrix*}
    \quad \mbox{ for } j, k = 1,2, \ldots, d
    \mbox{ and } j \neq k
  \end{equation}
  and with the vector $g^*$ being defined by
  \begin{equation}
    g^*(t) = (\underbrace{0, \ldots, 0}_{r_1 - 1\text{ times}}, g_1(t),
              \underbrace{0, \ldots, 0}_{r_2 - 1\text{ times}}, g_2(t),
              \ldots 
              \underbrace{0, \ldots, 0}_{r_d - 1\text{ times}}, g_d(t))^\top . 
  \end{equation}
\end{subequations}
While the dimension $d^*$ of this new system is potentially very much larger than the dimension $d$ of the original system, 
thus substantially increasing the complexity, 
we obtain a significant advantage because all equations of the system now have the same order, 
so that we may invoke the well known classical theory to investigate the asymptotic behavior of solutions of the system. 
Specifically, in view of this construction, we can immediately deduce from \cite[Theorem~8.1]{Diethelm2010}:

\begin{theorem}
  \label{thm:system-commensurate}
  Let the function $g: [0, \infty) \to \C^d$ be continuous and satisfy $g(t) \to 0$ for $t \to \infty$.
  Moreover, assume that $\alpha_j \in (0,1]$ for all $j$
  and that there exist some $\alpha^* \in (0,1]$ and some $\rho_j \in \Q$ such that $\alpha_j = \rho_j \alpha^*$ for all $j$.
  Then, all solutions of the original differential equation system \eqref{eq:system} converge to zero at infinity
  if the eigenvalues $\lambda_j^*$ of the associated system's coefficient matrix $A^*$
  defined in eqs.\ \eqref{eq:commensurate-matrix-a}, \eqref{eq:commensurate-matrix-b} and \eqref{eq:commensurate-matrix-c}
  satisfy $| \arg \lambda_j^* | > \pi \alpha^* / (2q)$ for all $j$, where $q$ is the least common multiple of the
  denominators of the $\rho_j$.
\end{theorem}

\begin{proof}
  From \cite[Theorem~8.1]{Diethelm2010}, we see that the systems
  \eqref{eq:system} and \eqref{eq:commensurate-matrix} are equivalent. Hence,
  we only concentrate on the system \eqref{eq:commensurate-matrix}. By changing
  variable $x^*=Ty$, where $T$ is the non-singular matrix which transforms $A^*$
  into a Jordan normal form $B$, the system \eqref{eq:commensurate-matrix}
  becomes 
  \begin{equation}\label{eq_equivalent}
    D_*^{\alpha^* / q} y(t) = B y(t) + \hat{g}(t),
  \end{equation}
  where $B = T^{-1} A T = \mathop{\text{diag}} (B_1, \dots, B_j, \dots, B_s)$ 
  where $B_j$ is the
  Jordan block corresponding the eigenvalue $\lambda^*_j$ of the matrix $A^*$
  and $\hat{g} = T^{-1} g^*$. Note that $\lim_{t \to \infty} \hat{g}(t) = 0$. Now,
  using the same arguments as in the proof of Theorem \ref{thm:characterSpec}
  and Corollary \ref{cor:inhomogeneous}, we see that every solution of the system
  \eqref{eq:commensurate-matrix} tends to zero if and only if the
  eigenvalues $\lambda^*_j$ of the associated system's coefficient matrix
  $A^*$ satisfy $| \arg \lambda^*_j | > \pi \alpha^* / (2q)$ for all $j$. The
  proof is complete. 
\end{proof}

Unfortunately, this criterion is based on the new system's coefficient matrix $A^*$, and thus it only indirectly makes use of the
coefficients of the original matrix $A$. It would be useful to have a formulation that allows to directly draw such a conclusion
from the original matrix without having to explicitly form the much larger new matrix and to compute its eigenvalues.
However, the following example indicates that we can probably not expect to find a simple criterion that permits to immediately decide the
question for the solution asymptotics for a given differential equation system.

\begin{example}
  \label{ex:1}
  Consider the system
  \begin{equation}\label{cex2}
    \begin{pmatrix}
      D_{*}^{1/2} x_1(t)
      \\
      D_{*}^{1/4} x_2(t)
    \end{pmatrix}
    =
    A x(t)
    \quad \mbox{ where }
    A 
    = 
    \begin{pmatrix}
      a_{11} & a_{12} \\
      a_{21} & a_{22}
    \end{pmatrix}
    = \begin{pmatrix*}[l]
        \phantom{-}0.00001 &  1 \\
                  -0.0022  &  0.1
    \end{pmatrix*}.
  \end{equation}
  Following the development above, we may choose $\alpha^* = 1$ and $q = 4$ in this example, 
  and thus this two-dimensional system can be rewritten as a three-dimensional system of order $\alpha^*/q = 1/4$ in the form
  \begin{equation}
    D^{1/4} x^*(t) 
    =
    A^* x^*(t)
    \quad \mbox{ with } 
    A^* 
    =
    \begin{pmatrix*}[l]
      \phantom{-}0       & 1  &  0 \\         
      \phantom{-}0.00001 & 0  &  1 \\
                -0.0022  &  0 &  0.1
    \end{pmatrix*}.
  \end{equation}
  The components $x_1^*$ and $x_3^*$ of the solution to this new system are then identical to the two components 
  $x_1$ and $x_2$, respectively, of the original system's solution.
  The eigenvalues of $A^*$ are $\lambda_1^* = -0.103917$ and  $\lambda_{2/3}^* = 0.101958 \pm 0.10385 \mathrm{i}$
  so that $\arg \lambda_1^* = \pi$ and $|\arg \lambda_2^*| = |\arg \lambda_3^*| = 0.79459 > \pi/8 = \pi \alpha^* /(2q)$.
  Therefore, Theorem \ref{thm:system-commensurate} asserts that all solutions of the system given in eq.\ \eqref{cex2} tend to zero at infinity.

  However, this observation does not appear to be immediately deducible from the original matrix $A$. By a simple calculation, we see that the eigenvalues of this matrix are $\lambda_1 = 0.0673111$ and $\lambda_2 = 0.0326989$ and thus
  $\arg \lambda_1 = \arg \lambda_2 = 0$ --- a property that one would
  normally associate with a system for which, in particular,
  unbounded solutions must be expected.

  Similarly, the diagonal entries of $A$ are real and positive as well, so their arguments are zero too.
  Thus, an argumentation based on the diagonal entries and not the eigenvalues 
  like the one that we had shown to be valid for triangular systems in Subsection \ref{subs:triangular}
  is not directly applicable to the case of a general (non-triangular) coefficient matrix either.
\end{example}

This seemingly negative observation is not the final word though.
Using different techniques we may actually derive a strategy that allows to investigate the stability question in a satisfactory manner
at least for the case of a homogeneous system.
Specifically, from the proof of Theorem \ref{exist_Theorem} 
we see that all solutions of the homogeneous multi-order system \eqref{eq:system2hom} are exponentially bounded. 
(This essentially follows from the generalized power series representation of the solution components and the estimate \eqref{eq:bound-beta} for
the coefficients of these series.)
Hence, we may take the Laplace transform on both sides of this system. This leads to
\begin{equation}
  \label{tem_1:sys}
  s^{\alpha_i} X_i(s) - s^{\alpha_i-1} x_i(0) = \sum_{j=1}^d a_{ij} X_j(s), \quad i = 1,\ldots,d,
\end{equation}
where
$X_i(s)$ is the Laplace transform of the $i$-th component $x_i(t)$ of the solution $x(t)$. 
The system \eqref{tem_1:sys} can be rewritten in the form
\begin{subequations}
  \begin{equation}
    \label{tem_2:sys}
    \Delta(s) \cdot 
      \begin{pmatrix}
        X_1(s)\\
        X_2(s)\\
        \vdots\\
        X_d(s)
      \end{pmatrix}
    = 
    \begin{pmatrix}
      b_1(s)\\
      b_2(s)\\
      \vdots\\
      b_d(s)
    \end{pmatrix}
  \end{equation}
  where
  \[
    b_i(s) = s^{\alpha_i-1} x_i(0), \quad i=1,\dots,d,
  \]
  and
  \begin{equation}
    \label{eq:def-Delta}
    \Delta(s)
    =
    \begin{pmatrix}
      s^{\alpha_1} - a_{11} & -a_{12}              & \cdots & -a_{1d} \\
       -a_{21}           & s^{\alpha_2} - a_{22} & \cdots & -a_{2d} \\
            \vdots      &     \ddots          &  \ddots  & \vdots \\
         -a_{d1}         &      \cdots        & - a_{dd-1} & s^{\alpha_d} - a_{dd}
    \end{pmatrix} 
    =
    \mathop{\mathrm{diag}} (s^{\alpha_1}, \ldots, s^{\alpha_d}) - A.
  \end{equation}
\end{subequations}
Using a standard result from the Laplace transform based stability theory \cite[Theorem 1]{DLL_2007}, 
we immediately obtain the following criterion on the asymptotic behavior of the system \eqref{eq:system2hom}:

\begin{theorem}\label{Stable}
  Consider the homogeneous multi-order system \eqref{eq:system2hom} 
  and let the function $\Delta$ be defined as in \eqref{eq:def-Delta}.
  If all the roots of the characteristic equation $\det \Delta(s) = 0$ have negative real parts, 
  then all solutions of the system \eqref{eq:system2hom} converge to zero at infinity.
\end{theorem}

\begin{remark}
  In the triangular case considered in Subsection \ref{subs:triangular}, we were able to extend our results
  derived for homogeneous equations also to the inhomgeneous case, cf.\ Corollary \ref{cor:inhomogeneous}. 
  This was possible mainly because the triangular structure allowed us to handle the individual equations
  of the given system in a step-by-step manner one at a time which made it possible to employ the 
  variation-of-constants formula that is available for scalar equations or single-order systems.
  In the general case considered here, a suitable generalization of the variation-of-constants formula
  to the setting of multi-order systems is not readily available and does not appear to be straightforward
  to derive. The authors plan to address this question in a future work.
\end{remark}


\appendix
\section{Auxiliary results}
\label{sec:appendix}

In this appendix we collect some auxiliary results that we used in the proofs of our theorems above.
For the formulation of these auxiliary results we shall use the notation
\begin{equation*}
  \Lambda^{\mathrm s}_\alpha \coloneqq \left\{ \lambda \in \C \setminus \{ 0 \} : |\arg{(\lambda)}| > \frac {\alpha \pi} {2} \right\}
\end{equation*}
and
\begin{equation*}
  \Lambda^{\mathrm u}_\alpha \coloneqq \left\{ \lambda \in \C \setminus \{ 0 \} : |\arg{(\lambda)}| < \frac {\alpha \pi} {2} \right\}
\end{equation*}
where the superscripts ``s'' and ``u'' can be interpreted as ``stable region'' and ``unstable region'', respectively.
We note that the lemmas below can be interpreted as generalizations of some results provided in \cite{Cong_2} where similar statements
have been derived under more restrictive assumptions on the parameter $\lambda$.

\begin{lemma} \label{lemma3}
  Let $\lambda$ be an arbitrary complex number and $\alpha\in (0,1]$. 
  There exists a positive real number $m(\alpha, \lambda)$ such that for every $t > 0$ the following estimates hold:
  \begin{itemize}
  \item [(i)] If $\lambda \in \Lambda_\alpha^{\mathrm u}$ then 
    \begin{align*} 
      \left| E_\alpha(\lambda t^{\alpha}) - \frac{1}{\alpha} \exp{(\lambda^{1/\alpha} t)} \right| 
      & \le m(\alpha, \lambda) \min \{ t^{-\alpha}, 1 \} ,\\
      \left| t^{\alpha-1} E_{\alpha, \alpha} (\lambda t^{\alpha}) - \frac{1}{\alpha} \lambda^{1/\alpha-1} \exp{(\lambda^{1/\alpha} t)} \right| 
      & \le m(\alpha, \lambda) \min \{ t^{-1-\alpha}, t^{-1+\alpha} \}.
    \end{align*}
  \item [(ii)] If $\lambda \in \Lambda_\alpha^{\mathrm s}$ then 
    \[ \left| t^{\alpha-1} E_{\alpha, \alpha} (\lambda t^{\alpha}) \right| \le m(\alpha, \lambda) \min \{ t^{-1-\alpha}, t^{-1+\alpha} \}. \]
  \end{itemize}
\end{lemma}

\begin{proof}
  In the case $\alpha = 1$ the results are trivially true because then $E_\alpha(z) = E_{\alpha, \alpha}(z) = \exp(z)$. 
  We therefore only have to deal with the case $0 < \alpha < 1$ explicitly.

  Let us start with the case $0 < t \le 1$. In this case, the minimum in the first claim of (i) has the value~$1$. 
  Thus, this claim is an immediate consequence of the fact that the expression on its
  left-hand side is a continuous function of $t \in [0,1]$. Similarly, we can see --- in view of the continuity of the
  Mittag-Leffler functions and the exponentials on $[0,1]$ --- that the expressions on the left-hand sides of the two other claims
  can be bounded by $O(t^{-1+\alpha}) = O(\min \{ t^{-1-\alpha}, t^{-1+\alpha} \})$.  

  The statements for $t > 1$ (where the minima are always attained by the first expression in the braces) 
  immediately follow from well-known results about the asymptotic behavior of Mittag-Leffler functions; 
  specifically, we have (cf., e.g., \cite[Proposition~3.6 and Theorem~4.3]{GorenfloEtal2014} or \cite[Theorems~1.3 and~1.4]{Podlubny1999})
  that
  \begin{equation}
   \label{eq:proof-3a}
    E_{\alpha, \beta}(z) 
    = 
   \frac 1 \alpha z^{(1-\beta)/\alpha} \exp( z^{1/\alpha}) - \sum_{k=1}^p \frac {z^{-k}} {\Gamma(\beta - \alpha k)} 
   + O( |z|^{-p-1} )
   \qquad 
   \mbox{ for }
   z \in \Lambda_\alpha^{\mathrm u}
  \end{equation}
  and
  \begin{equation}
    \label{eq:proof-3b}
    E_{\alpha, \beta}(z) 
    = 
    - \sum_{k=1}^p \frac {z^{-k}} {\Gamma(\beta - \alpha k)} 
   + O( |z|^{-p-1} )
   \qquad 
   \mbox{ for }
   z \in \Lambda_\alpha^{\mathrm s}
  \end{equation}
  hold for arbitrary $p \in \N$ and $|z| \to \infty$.
  Upon choosing $t > 0$ and $z \coloneqq \lambda t^\alpha$, we then observe that the relation $z \in \Lambda_\alpha^{\mathrm s}$ holds 
  if and only if $\lambda  \in \Lambda_\alpha^{\mathrm s}$, and an analog equivalence exists for $\Lambda_\alpha^{\mathrm u}$.
  Using this approach, the first statement of (i) follows from eq.\ \eqref{eq:proof-3a} with $p=1$. 
  Similarly, the second statement of (i) and the statement of (ii) follow from eqs.\ \eqref{eq:proof-3a} and \eqref{eq:proof-3b}, respectively,
  upon setting $p=2$ and noticing that the summands for $k=1$ vanish because they contain a factor $1 / \Gamma(\alpha - \alpha) = 1 / \Gamma(0) = 0$.
\end{proof}

\begin{lemma}\label{lemma4}
  Let $\lambda \in \C \setminus \{ 0 \}$ and $\alpha \in (0,1]$. 
  There exists a positive constant $K(\alpha, \lambda)$ such that for all $t \ge 1$ the following estimates hold: 
  \begin{itemize}
  \item [(i)] If $\lambda\in\Lambda_\alpha^{\mathrm u}$ then
    \begin{align*}
      \int_t^\infty \left| \lambda^{1/\alpha-1} E_\alpha(\lambda t^\alpha) \exp(-\lambda^{1/\alpha} \tau) \right| \, d \tau
      & \leq K(\alpha, \lambda),\\
      \int_0^t \left|
      \left( (t-\tau)^{\alpha-1} E_{\alpha,\alpha}(\lambda(t-\tau)^{\alpha}) - \lambda^{1/\alpha-1} E_\alpha(\lambda t^\alpha) \exp(-\lambda^{1/\alpha} \tau)\right)
      \right| \, d \tau
      & \leq K(\alpha, \lambda).
    \end{align*}
  \item [(ii)] If $\lambda \in \Lambda_\alpha^{\mathrm s}$ then
    \[
    \int_0^t \left| (t-\tau)^{\alpha-1} E_{\alpha,\alpha}(\lambda (t-\tau)^{\alpha}) \right| \, d \tau \leq K(\alpha, \lambda).
    \]
  \end{itemize}
\end{lemma}

\begin{proof}
  Once again the statements are trivially true for $\alpha = 1$. 
  The proof of the remaining cases is very similar to the proof of \cite[Lemma 5]{Cong_2}.

  For the first claim of part (i), the first statement of Lemma \ref{lemma3}(i) allows us to proceed as follows: 
  \begin{align*}
    \int_t^\infty \left| \lambda^{1/\alpha-1} E_\alpha(\lambda t^\alpha) \exp(-\lambda^{1/\alpha} \tau) \right| \, d \tau
      & \le   | \lambda |^{1/\alpha-1} \int_t^\infty  \left( \left| \frac 1 \alpha \exp(\lambda^{1/\alpha} t) \right|
          + \frac {m(\alpha, \lambda)} { t^\alpha} \right) \left| \exp(-\lambda^{1/\alpha} \tau) \right| \, d \tau \\
      &  =  | \lambda |^{1/\alpha-1} \frac 1 \alpha \int_t^\infty \left| \exp(\lambda^{1/\alpha} (t-\tau)) \right| \, d \tau \\
          &  \phantom{= } {}      + | \lambda |^{1/\alpha-1} \frac {m(\alpha, \lambda)} {t^\alpha} \int_t^\infty \left| \exp(-\lambda^{1/\alpha} \tau) \right| \, d \tau.
  \end{align*}
  For the evaluation of these integrals we recall that $\lambda \in \Lambda_\alpha^{\mathrm u}$, 
  and hence $|\arg \lambda^{1/\alpha}| < \pi/2$ which implies that $\Re \lambda^{1/\alpha} > 0$. 
  Making use of this inequality in combination with the identity $|\exp(\lambda^{1/\alpha} z)| = \exp( \Re \lambda^{1/\alpha} z)$ for $z \in \R$, we conclude
  \begin{equation*}
    \int_t^\infty \left| \exp(\lambda^{1/\alpha} (t-\tau)) \right| \, d \tau 
    = 
    \int_{-\infty}^0  \exp( \Re \lambda^{1/\alpha} u) \, d u 
    = \frac 1 {\Re \lambda^{1/\alpha}}
  \end{equation*}
  and
  \begin{equation*}
     \int_t^\infty \left| \exp(-\lambda^{1/\alpha} \tau) \right| \, d \tau
     = \int_t^\infty \exp(- \Re \lambda^{1/\alpha} \tau) \, d \tau 
     < \int_0^\infty \exp(- \Re \lambda^{1/\alpha} \tau) \, d \tau 
     = \frac 1 {\Re \lambda^{1/\alpha}}.
  \end{equation*}
  These estimates conclude this part of the proof.

  The proof of the second claim of part (i) uses the second statement of Lemma \ref{lemma3}(i).
  Specifically, that result allows us to write
  \begin{align}
    \nonumber
    \MoveEqLeft \int_0^t \left|
     (t-\tau)^{\alpha-1} E_{\alpha,\alpha}(\lambda(t-\tau)^{\alpha}) - \lambda^{1/\alpha-1} E_\alpha(\lambda t^\alpha) \exp(-\lambda^{1/\alpha} \tau)
    \right| \, d \tau  \\
    \label{eq:proof4a}
    \le &
      \int_0^t \left| \frac 1 \alpha \lambda^{1/\alpha - 1} \exp(\lambda^{1/\alpha} (t - \tau)) 
                        - \lambda^{1/\alpha-1} E_\alpha(\lambda t^\alpha) \exp(-\lambda^{1/\alpha} \tau ) \right| \, d \tau \\
    \nonumber
      & 
      {} + m(\alpha, \lambda) \int_0^t \min \{ (t-\tau)^{-1-\alpha}, (t-\tau)^{-1+\alpha} \} \, d \tau
  \end{align}
  Since we have assumed that $t \ge 1$, we may bound the last integral as follows:
  \begin{align}
    \nonumber
    \int_0^t \min \{ (t-\tau)^{-1-\alpha}, (t-\tau)^{-1+\alpha} \} \, d \tau 
    & = 
    \nonumber
    \int_0^t \min \{ \tau^{-1-\alpha}, \tau^{-1+\alpha} \} \, d \tau \\
    \nonumber
    & =  
    \int_0^1 \tau^{-1+\alpha} \, d \tau + \int_1^t \tau^{-1-\alpha} \, d \tau \\
    \nonumber
    & = \frac 1 \alpha + \frac 1 {-\alpha} \left( t^{-\alpha} - 1 \right) \\
    \label{eq:proof4b}
    & = \frac 2 \alpha - \frac 1 \alpha t^{-\alpha} 
     <  \frac 2 \alpha.
  \end{align}
  Moreover, for the first integral on the right-hand side of eq.\ \eqref{eq:proof4a} we may invoke the first statement of Lemma \ref{lemma3}(i) and conclude that
  \begin{align*}
    \MoveEqLeft
    \int_0^t \left| \frac 1 \alpha \lambda^{1/\alpha - 1} \exp(\lambda^{1/\alpha} (t - \tau)) 
                        - \lambda^{1/\alpha-1} E_\alpha(\lambda t^\alpha) \exp(-\lambda^{1/\alpha} \tau ) \right| \, d \tau  \\
    & = 
        |\lambda|^{1/\alpha - 1} \left| \frac 1 \alpha \exp(\lambda^{1/\alpha} t) - E_\alpha(\lambda t^\alpha) \right| 
            \int_0^t | \exp (-\lambda^{1/\alpha} \tau) | \, d \tau \\
    & \le  |\lambda|^{1/\alpha - 1} m(\alpha, \lambda) t^{-\alpha}  \int_0^t | \exp (-\lambda^{1/\alpha} \tau) | \, d \tau \\
    & =    |\lambda|^{1/\alpha - 1} m(\alpha, \lambda) t^{-\alpha} \int_0^t  \exp (- \Re \lambda^{1/\alpha} \tau)  \, d \tau \\
    & =    |\lambda|^{1/\alpha - 1} m(\alpha, \lambda) t^{-\alpha} \frac 1 {\Re \lambda^{1/\alpha}} \left ( 1 - \exp (-\Re \lambda^{1/\alpha} t) \right ).
  \end{align*}
  As above, our assumption that $\lambda \in \Lambda_\alpha^{\mathrm u}$ implies that $\Re \lambda^{1/\alpha} > 0$, and hence this last expression 
  is uniformly bounded for all $t \ge 1$. This completes the proof of the second statement of part (i).

  Finally, for part (ii), Lemma \ref{lemma3}(ii) and the fact that $t \ge 1$ allow us to estimate as follows:
  \begin{align*}
    \int_0^t \left| (t-\tau)^{\alpha-1} E_{\alpha,\alpha}(\lambda (t-\tau)^{\alpha}) \right| \, d \tau
       & =    \int_0^t \left| \tau^{\alpha-1} E_{\alpha,\alpha}(\lambda \tau^{\alpha}) \right| \, d \tau \\
       & \le  m(\alpha, \lambda) \int_0^t \min \{ \tau^{-1-\alpha}, \tau^{-1+\alpha} \} \, d \tau \\
       & <    m(\alpha, \lambda) \frac 2 \alpha
  \end{align*}
  where the last estimate uses the result \eqref{eq:proof4b}. Thus the desired result follows.
\end{proof}

\begin{lemma}\label{LimitLemma}
  For any continuous and bounded function $f: [0, \infty) \rightarrow \C$, $\alpha \in (0,1]$ and $\lambda \in \Lambda_\alpha^{\mathrm u}$, we have
  \begin{equation}\label{Eq4}
    \lim_{t\to\infty} \int_0^t (t-\tau)^{\alpha-1} \frac{E_{\alpha,\alpha}(\lambda (t-\tau)^\alpha)}{E_\alpha(\lambda t^\alpha)} f(\tau) \, d \tau
    = \lambda^{1/\alpha-1} \int_0^\infty \exp(-\lambda^{1/\alpha} \tau) f(\tau) \, d \tau.
  \end{equation}
\end{lemma}

\begin{proof}
  Again, the case $\alpha = 1$ is trivial.

  For $0 < \alpha < 1$, we first remark that the expression on the left-hand side of eq.\ \eqref{Eq4} is well defined: 
  The denominator is non-zero because, as shown by Wiman \cite[pp.~225--226]{Wiman}, 
  the Mittag-Leffler function $E_\alpha$ does not have any zeros in $\Lambda_\alpha^{\mathrm u}$. 
  Thus, since $\lambda \in \Lambda_\alpha^{\mathrm u}$ implies that $t \lambda \in \Lambda_\alpha^{\mathrm u}$ for all $t > 0$, 
  we conclude that $E_\alpha(\lambda t^\alpha) \ne 0$ for all $t > 0$.

  Next we note that the integral on the right-hand side of eq.\ \eqref{Eq4} exists because $f$ is assumed to be continuous (which asserts the existence
  of the integral over any compact subinterval $[0,T]$ with arbitrary $T > 0$) and bounded which admits us to bound the absolute value of the integrand 
  by
  \begin{equation*}
    |\exp(-\lambda^{1/\alpha} \tau) f(\tau)| \le \exp (-\Re \lambda^{1/\alpha} \tau) \cdot \sup_{t \ge 0} | f(t) | .
  \end{equation*}
  As we already noted in earlier proofs, $\Re \lambda^{1/\alpha} > 0$, and hence this bound provides a convergent majorant for the integral over $[0, \infty)$, 
  thus asserting the existence and finiteness of the improper integral on the right-hand side of eq.\ \eqref{Eq4}.

  Then, the first statement of Lemma \ref{lemma3}(i) implies that $| E_\alpha (\lambda t^\alpha) |$ exhibits an unbounded growth as $t \to \infty$ and
  hence that
  \begin{equation*}
    \lim_{t \to \infty}  \frac {\frac 1 \alpha \exp(\lambda^{1/\alpha} t) } {E_\alpha(\lambda t^\alpha)} = 1.
  \end{equation*}
  It thus follows that 
  \begin{equation*}
    \lim_{t\to\infty} \int_0^t (t-\tau)^{\alpha-1} \frac{E_{\alpha,\alpha}(\lambda (t-\tau)^\alpha)}{E_\alpha(\lambda t^\alpha)} f(\tau) \, d \tau
    =  \lim_{t\to\infty} \alpha \int_0^t (t-\tau)^{\alpha-1} \frac{E_{\alpha,\alpha}(\lambda (t-\tau)^\alpha)}{\exp(\lambda^{1/\alpha} t)} f(\tau) \, d \tau
  \end{equation*}
  if one of the limits exists (which immediately implies the existence of the other one). 

  For $t > 1$ we see that
  \begin{align*}
    \left | \alpha  \int_{t-1}^t (t-\tau)^{\alpha-1} E_{\alpha,\alpha}(\lambda (t-\tau)^\alpha) f(\tau) \, d \tau \right|
    & \le   \sup_{u \ge 0} | f(u) | \cdot \sup_{0 \le u \le 1} | E_{\alpha,\alpha}(\lambda u^\alpha) | \cdot \alpha \int_0^1 u^{\alpha-1} \, d u \\
    & =     \sup_{u \ge 0} | f(u) | \cdot \sup_{0 \le u \le 1} | E_{\alpha,\alpha}(\lambda u^\alpha) | .
  \end{align*}
  Evidently, the upper bound depends on $f$, $\alpha$ and $\lambda$ but not on $t$. 
  It therefore follows, once again using the unbounded growth of $|\exp (\lambda^{1/\alpha} t)|$ for $t \to \infty$, that
  \begin{equation*}
    \lim_{t \to \infty} \alpha  \int_{t-1}^t (t-\tau)^{\alpha-1} \frac {E_{\alpha,\alpha}(\lambda (t-\tau)^\alpha)}{\exp(\lambda^{1/\alpha} t)} f(\tau) \, d \tau 
    = 0. 
  \end{equation*}

  In order to complete the proof of Lemma \ref{LimitLemma}, it therefore suffices to show that
  \begin{equation}
    \label{eq:proof-limitlemma}
    \lim_{t\to\infty} \alpha \int_0^{t-1} (t-\tau)^{\alpha-1} \frac{E_{\alpha,\alpha}(\lambda (t-\tau)^\alpha)}{\exp(\lambda^{1/\alpha} t)} f(\tau) \, d \tau
     = \lambda^{1/\alpha-1} \int_0^\infty \exp(-\lambda^{1/\alpha} \tau) f(\tau) \, d \tau.
  \end{equation}
  To this end, we recall that the second statement of Lemma \ref{lemma3}(i) implies 
  \begin{align*}
    \MoveEqLeft
      \left| \int_0^{t-1} \frac {\alpha (t - \tau)^{\alpha-1} E_{\alpha, \alpha} (\lambda (t - \tau)^\alpha) - \lambda^{1/\alpha - 1} \exp( \lambda^{1/\alpha} (t - \tau) )}
                             {\exp (\lambda^{1/\alpha} t)} f(\tau) \, d \tau \right|  \\
    & \le 
    \sup_{u \ge 0} | f(u) | \cdot 
      \left| \int_0^{t-1} \frac {\alpha m(\alpha, \lambda) (t-\tau)^{-1-\alpha}} {\exp (\lambda^{1/\alpha} t)} \, d \tau \right| \\
    & \le      
    \sup_{u \ge 0} | f(u) | \frac {\alpha m(\alpha, \lambda)} {| \exp (\lambda^{1/\alpha} t) |} \int_1^t \tau^{-1-\alpha} \, d \tau
     \le  \sup_{u \ge 0} | f(u) | \frac {m(\alpha, \lambda)} {| \exp (\lambda^{1/\alpha} t) |} 
  \end{align*}
  for $t > 1$; in particular we once again see that the upper bound converges to zero as $t \to \infty$, 
  and therefore \eqref{eq:proof-limitlemma} follows as desired.
\end{proof}

Using Lemma \ref{lemma3}, Lemma \ref{lemma4} and Lemma \ref{LimitLemma}, 
we obtain the asymptotic behavior of solutions to scalar linear fractional differential equations as follows.

\begin{lemma}\label{stableLemma}
  Let $\alpha \in (0,1]$, and let $f: [0,\infty) \rightarrow \C$ be a continuous function with the property $\lim_{t\to\infty} |f(t)| = 0$. 
  Consider the differential equation
  \begin{equation}\label{scalarEq}
    D_*^\alpha x(t) = \lambda x(t) + f(t), \qquad t > 0.
  \end{equation}
  The following statements hold:
  \begin{itemize}
  \item [(i)] If $|\arg{(\lambda)}| > \alpha \pi / 2$ then all solutions of \eqref{scalarEq} tend to zero as $t \to \infty$.
  \item[(ii)] If $|\arg{(\lambda)}| < \alpha \pi / 2$ then eq.\ \eqref{scalarEq} has a unique bounded solution. 
    Moreover, this solution tends to zero as $t \to \infty$.
  \end{itemize}
\end{lemma}

\begin{proof}
  In either case, we start from the variation of constants formula \cite[Theorem~7.2 and Remark~7.1]{Diethelm2010}
  which tells us that the solution $\varphi(\cdot, x_0)$ of \eqref{scalarEq} that satisfies the condition $\varphi(0, x_0) = x_0$
  is given by
  \begin{equation}
    \label{eq:var-of-const}
    \varphi(t,x_0)
    = x_0 E_\alpha (\lambda t^\alpha) + \int_0^t (t-\tau)^{\alpha-1} E_{\alpha,\alpha}(\lambda(t-\tau)^\alpha) f(\tau) \, d \tau .
  \end{equation}

  In order to prove part (i), let $\varepsilon > 0$ be arbitrarily small. 
  We can find a constant $T > 0$ such that $|f(t)| < \varepsilon$ for all $t \ge T$. 
  For $t > T + 1$ and $x_0 \in \C$, we split up the integral on the right-hand side of eq.\ \eqref{eq:var-of-const} according to
  \begin{align*}
    \varphi(t,x_0)
    &  = x_0 E_\alpha (\lambda t^\alpha) + \int_0^T (t-\tau)^{\alpha-1} E_{\alpha,\alpha}(\lambda(t-\tau)^\alpha) f(\tau) \, d \tau
         + \int_T^{t-1} (t-\tau)^{\alpha-1} E_{\alpha,\alpha}(\lambda(t-\tau)^\alpha) f(\tau) \, d \tau \\
    & \phantom{= x_0 E_\alpha (\lambda t^\alpha)} {}  + \int_{t-1}^t (t-\tau)^{\alpha-1} E_{\alpha,\alpha}(\lambda(t-\tau)^\alpha) f(\tau)\, d \tau.
  \end{align*} 
  By virtue of Lemma \ref{lemma3}(ii), we have
  \begin{equation}\label{ScalarEq1+}
    \lim_{t\to \infty} x_0 E_\alpha(\lambda t^\alpha) = 0.
  \end{equation}
  On the other hand, by a simple computation, we obtain
  \begin{equation}\label{ScalarEq2+}
    \left| \int_T^{t-1} (t-\tau)^{\alpha - 1} E_{\alpha, \alpha} (\lambda (t-\tau)^\alpha) f(\tau) \, d \tau \right| 
     \le \varepsilon \int_{1}^{t-T} | \tau^{\alpha - 1} E_{\alpha,\alpha} (\lambda \tau^\alpha) | \, d \tau 
     \le \frac{\varepsilon\, m(\alpha,\lambda)}{\alpha}
  \end{equation}
  due to Lemma~\ref{lemma3}(ii)
  and 
  \begin{equation}\label{ScalarEq3+}
    \left| \int_{t-1}^t (t-\tau)^{\alpha-1} E_{\alpha, \alpha} (\lambda (t-\tau)^\alpha) f(\tau) \, d \tau \right| 
    \le \varepsilon \int_0^1 | \tau^{\alpha - 1} E_{\alpha, \alpha} (\lambda \tau^\alpha) | \, d \tau 
    \le \varepsilon  E_{\alpha, \alpha+1}(| \lambda |)
  \end{equation}
  (see \cite[eq.\ (1.99)]{Podlubny1999}).
  Furthermore, 
  \begin{equation}\label{ScalarEq4+}
    \left| \int_0^T (t-\tau)^{\alpha-1} E_{\alpha,\alpha}(\lambda (t-\tau)^\alpha) f(\tau) d \tau \right| 
    \le \sup_{t \geq 0} |f(t)| \int_{t-T}^t | \tau^{\alpha-1} E_{\alpha, \alpha} (\lambda \tau^\alpha) | \, d \tau
    \le \frac{m(\alpha,\lambda) \sup_{t \geq 0} |f(t)|}{\alpha (t-T)^\alpha}
  \end{equation}
  due to Lemma~\ref{lemma3}(ii).
  Since $\epsilon$ is arbitrarily small, from eqs.\ \eqref{ScalarEq1+}, \eqref{ScalarEq2+}, \eqref{ScalarEq3+} and \eqref{ScalarEq4+}, we get
  \[
  \lim_{t\to \infty} |\varphi(t,x_0)| = 0,
  \]
  and the proof of part (i) is complete.
  
  For the proof of (ii), we note that Lemma \ref{LimitLemma} admits us to precisely describe the asympotic behavior of the 
  integral on the right-hand side of eq.\ \eqref{eq:var-of-const}, namely
  \begin{equation*}
    \int_0^t (t-\tau)^{\alpha-1} E_{\alpha,\alpha}(\lambda(t-\tau)^\alpha) f(\tau) \, d \tau
    = E_\alpha (\lambda t^\alpha) \lambda^{1/\alpha - 1} \int_0^\infty \exp( - \lambda^{1/\alpha} \tau) f(\tau) \, d \tau \cdot (1 + o(1)).
  \end{equation*}
  Thus, by \eqref{eq:var-of-const}, any solution to the differential equation behaves as
  \begin{equation}
    \label{eq:ScalarEq5+}
    \varphi(t, x_0)
    = E_\alpha (\lambda t^\alpha) \left[ x_0 + \lambda^{1/\alpha - 1} \int_0^\infty \exp( - \lambda^{1/\alpha} \tau) f(\tau) \, d \tau \cdot (1 + o(1)) \right]
  \end{equation}
  for $t \to \infty$. Since $| \arg \lambda | < \alpha \pi /2$, we know that $E_\alpha (\lambda t^\alpha)$ is unbounded as $t \to \infty$. 
  Thus, a necessary condition for the entire expression on the right-hand side of \eqref{eq:ScalarEq5+} to be bounded is that the term in brackets converges to zero
  as $t \to \infty$. Clearly, this is the case if and only if 
  \begin{equation*}
    x_0 =  \overline{x}_0 \coloneqq - \lambda^{1/\alpha - 1} \int_0^\infty \exp( - \lambda^{1/\alpha} \tau) f(\tau) \, d \tau.
  \end{equation*}
  Thus, the differential equation \eqref{scalarEq} has at most one bounded solution, and it remains to prove that this solution has the property 
  $\varphi(t, \overline{x}_0) \to 0$ as $t \to \infty$ (which, in particular, implies that the solution is bounded and hence that a bounded solution exists).

  To this end, let $\varepsilon > 0$ be an arbitrary positive real number. Then there exists a positive constant $T > 0$ such that 
  \begin{equation}\label{DecayCond}
    |f(t)| \le \varepsilon \qquad \mbox{ for all } t \geq T.
  \end{equation}
  For any $t \geq T+ 1$, we put
  \begin{align*}
    H_1(t) & = - E_\alpha(\lambda t^\alpha) \lambda^{1/\alpha-1} \int_t^\infty \exp{(-\lambda^{1/\alpha} \tau)} f(\tau) \, d \tau , \\
    H_2(t) & = \int_0^T \left[ (t-\tau)^{\alpha-1} E_{\alpha, \alpha} (\lambda(t-\tau)^\alpha) 
                                - \lambda^{1/\alpha-1} \exp{(-\lambda^{1/\alpha}\tau)} E_\alpha(\lambda t^\alpha) \right] f(\tau) \, d \tau, \\
    H_3(t) & = \int_T^t \left[ (t-\tau)^{\alpha-1} E_{\alpha, \alpha} (\lambda(t-\tau)^\alpha) 
                                - \lambda^{1/\alpha-1} \exp{(-\lambda^{1/\alpha} \tau)} E_\alpha(\lambda t^\alpha)\right] f(\tau) \, d \tau.
  \end{align*}
  It is then clear from eq.\ \eqref{eq:var-of-const} and the definition of $\overline{x}_0$ that
  \begin{equation*}
    \varphi(t, \overline{x}_0) = H_1(t) + H_2(t) + H_3(t).
  \end{equation*}
  By virtue of \eqref{DecayCond} and the first statement of Lemma \ref{lemma4}(i), we have
  \begin{equation}\label{ScalarEq1-}
    |H_1(t)| \le  \varepsilon K(\alpha,\lambda).
  \end{equation}
  Using both statements of Lemma~\ref{lemma3}(i), we obtain, since $t - T \ge 1$,
  \begin{align}
    \nonumber
    |H_2(t)|
     \le &
       \sup_{t \geq 0} |f(t)| 
       \int_0^T \left[ 
         | \lambda |^{1/\alpha-1} \left| \frac 1 \alpha \exp(\lambda^{1/\alpha} (t - \tau)) - \exp(-\lambda^{1/\alpha} \tau) E_\alpha (\lambda t^\alpha) \right|
           + \frac {m(\alpha, \lambda)} {(t-\tau)^{1+\alpha}} \right]
          \, d \tau \\
     \nonumber
     \le &  
       \sup_{t \geq 0} |f(t)|  \left[ | \lambda |^{1/\alpha-1}
       \int_0^T | \exp(-\lambda^{1/\alpha} \tau)| \cdot \left|\frac 1 \alpha \exp(\lambda^{1/\alpha}t) - E_\alpha (\lambda t^\alpha) \right| \, d \tau 
           \right. \\
     \nonumber
       &  \phantom{ \sup_{t \geq 0} |f(t)| [ } \left . {} 
           + m(\alpha, \lambda) \int_0^T  \frac {d \tau} {(t-\tau)^{1+\alpha}}
       \right] \\
     \le & 
       m(\alpha, \lambda)  \sup_{t \geq 0} |f(t)|  
       \left[ | \lambda |^{1/\alpha-1} t^{-\alpha}
          \int_0^T | \exp(-\lambda^{1/\alpha} \tau)| \, d \tau
         + \frac {(t-T)^{-\alpha} - t^{-\alpha}} \alpha
       \right]   .
       \label{ScalarEq2-}
  \end{align}
  Since $\lambda \in \Lambda_\alpha^{\mathrm u}$, we conclude once again that
  \begin{equation*}
    \int_0^T | \exp(-\lambda^{1/\alpha} \tau)| \, d \tau
    =     \int_0^T \exp(-\Re \lambda^{1/\alpha} \tau) \, d \tau
    =     \frac 1 {\Re \lambda^{1/\alpha}} \left[ 1 - \exp ( -\Re \lambda^{1/\alpha} T ) \right] 
    \le  \frac 1 {\Re \lambda^{1/\alpha}}
  \end{equation*}
  and thus we see from eq.\ \eqref{ScalarEq2-} that 
  \begin{equation}
    \label{ScalarEq2a-}
    H_2(t) \to 0 \qquad \mbox{ as } t \to \infty.
  \end{equation}
  Furthermore, by \eqref{DecayCond} and the second statement of Lemma \ref{lemma4}(i), we have
  \begin{equation}\label{ScalarEq3-}
    |H_3(t)| \le \varepsilon K(\alpha,\lambda).
  \end{equation}
  From \eqref{ScalarEq1-}, \eqref{ScalarEq2a-}, \eqref{ScalarEq3-} and the
  fact that $\varepsilon > 0$ can be made arbitrarily small, we conclude
  \begin{equation*}
    \lim_{t\to \infty} \varphi (t, \overline{x}_0) = 0.
  \end{equation*}
  The proof is complete.
\end{proof}

\section*{Acknowledgement}
The work of H.T. Tuan is supported by the Vietnam National Foundation
for Science and Technology Development (NAFOSTED).

\end{document}